\documentclass[11pt]{article}
\usepackage{amssymb, latexsym, mathtools, amsthm}
\begingroup
    \makeatletter
    \@for\theoremstyle:=definition,remark,plain\do{%
        \expandafter\g@addto@macro\csname th@\theoremstyle\endcsname{%
            \addtolength\thm@preskip\parskip
            }%
        }
\endgroup
\usepackage{enumerate}
\usepackage{pgf,tikz}
\usepackage{pdfsync}
\usepackage{natbib}
\usepackage{txfonts}
\usepackage[T1]{fontenc}
\usepackage{booktabs}
 \usepackage{graphicx}
\definecolor{dnrbl}{rgb}{0,0,0.3}
\definecolor{dnrgr}{rgb}{0,0.3,0}
\definecolor{dnrre}{rgb}{0.5,0,0}
\usepackage[colorlinks=true, citecolor=dnrgr, linkcolor=dnrre, urlcolor=dnrre] {hyperref}
\usepackage{xcolor}
\usepackage{parskip}
\usepackage{tabularx, colortbl}
\usepackage{geometry}
\theoremstyle{plain}
\newtheorem{thm}{Theorem}[section]
\newtheorem{prop}[thm]{Proposition}

\newtheorem{lem}[thm]{Lemma}
\newtheorem{coro}[thm]{Corollary}

\theoremstyle{definition}

\newtheorem{defi}[thm]{Definition}
\makeatletter

\let\c@table\c@figure
\makeatother

\newcommand{\Nat}{\mathbb{N}}

\newcommand{\restr}{\upharpoonright}  
 
\newcommand{\un}{\uparrow} 
\newcommand{\de}{\downarrow} 
\DeclarePairedDelimiter{\ceil}{\lceil}{\rceil}

\DeclarePairedDelimiter{\dbra}{\llbracket}{\rrbracket}

\newcommand{\LL}{\mathcal{L}}
\newcommand{\LLast}{\mathcal{L}^{\ast}}

\newcommand{\XX}{\mathcal{X}}
\newcommand{\YY}{\mathcal{Y}}
\newcommand{\DD}{\mathcal{D}}
\newcommand{\dd}{\mathbf{d}}
\newcommand{\DDast}{\mathcal{D}^{\ast}}

\newcommand{\MM}{\mathcal{M}} 
\newcommand{\MMast}{\mathcal{M}^{\ast}} 
\newcommand{\CC}{\mathcal{C}}
\newcommand{\CCh}{\hat{\mathcal{C}}}

\newcommand{\wgt}[1]{\mathop{\mathtt{wgt}_n}\/\left({#1}\right)}

\newcommand{\PP}{\mathcal{P}}

\newcommand{\ml}{Martin-L\"{o}f }

\newcommand{\eg}{e.g.\ }
\newcommand{\ie}{i.e.\ }
\newcommand{\ce}{c.e.\ }

\newcommand{\pf}{prefix-free }

\newcommand{\VV}{\mathcal V}

\newcommand{\FF}{\mathcal F} 
 
\newcommand{\twome}{2^{\omega}} 
\newcommand{\twomel}{2^{<\omega}}  
   
\newcommand{\muast}{\mu^{\ast}}

\renewenvironment{abstract}
 { \normalsize
  \list{}{
    \setlength{\leftmargin}{.0cm}%
    \setlength{\rightmargin}{\leftmargin}%
    }%
  \item {\bf \abstractname.} \relax}
 {\endlist}
\usepackage{titlesec}
 \titleformat{\paragraph}
{\normalfont\normalsize\em\bfseries}{\theparagraph}{1em}{$\blacktriangleright$\hspace{0.2cm}}
\titlespacing*{\paragraph}
{0pt}{3.25ex plus 1ex minus .2ex}{0.5ex plus .2ex}

\title{Equivalences between learning of data and probability distributions,
and their applications
\thanks{Barmpalias was supported by the 
1000 Talents Program for Young Scholars from the Chinese Government No.~D1101130, 
NSFC grant No.~11750110425 and grant No.~ISCAS-2015-07 from the Institute of Software. 
Frank Stephan is supported in part by the Singapore Ministry
of Education Academic Research Fund MOE2016-T2-1-019 / R146-000-234-112.
The work was partially done while the authors were visiting (and being supported by) 
the Institute for Mathematical Sciences, National University of Singapore in 2017.}}

\author{George Barmpalias  \and Nan Fang \and Frank Stephan}
\date{\today}
\begin{document}
\maketitle
\begin{abstract}
Algorithmic learning theory traditionally studies the learnability of
effective infinite binary sequences (reals),
while recent work by   \citep{VITANYI201713}
and  \citep{Bienvenu2014} 
has adapted this framework to the study of learnability of effective 
probability distributions from  random data.
We prove that for certain families of probability 
measures that are parametrized by reals,
learnability of a subclass of probability measures is equivalent to learnability
of the class of the corresponding real parameters. This equivalence
allows to transfer results from classical algorithmic theory to
learning theory of probability measures. We present a number of such applications,
providing many new results regarding EX and BC learnability of classes of measures,
thus drawing parallels between the two learning theories. 
\end{abstract}
\vspace*{\fill}
\noindent{\bf George Barmpalias}\\[0.5em]
\noindent
State Key Lab of Computer Science, 
Institute of Software, Chinese Academy of Sciences, Beijing, China.\\[0.2em] 
\textit{E-mail:} \texttt{\textcolor{dnrgr}{barmpalias@gmail.com}}.
\textit{Web:} \texttt{\textcolor{dnrre}{http://barmpalias.net}}\par
\addvspace{\medskipamount}\medskip\medskip

\noindent{\bf Nan Fang}\\[0.2em]
\noindent Institut f\"{u}r Informatik, Ruprecht-Karls-Universit\"{a}t Heidelberg, Germany.\\[0.1em]
\textit{E-mail:} \texttt{\textcolor{dnrgr}{nan.fang@informatik.uni-heidelberg.de}.}
\textit{Web:} \texttt{\textcolor{dnrre}{http://fangnan.org}} \par
\addvspace{\medskipamount}\medskip\medskip

\noindent{\bf Frank Stephan}\\[0.5em]  
Department of Mathematics and
School of Computing,
National University of Singapore, Republic of Singapore.\\[0.2em]
\textit{E-mail:} \texttt{\textcolor{dnrgr}{fstephan@comp.nus.edu.sg.}}
\textit{Web:} \texttt{\textcolor{dnrre}{http://www.comp.nus.edu.sg/$\sim$fstephan}}

\vfill \thispagestyle{empty}
\clearpage

\section{Introduction}\label{N3sKx4pC1}
The present work concerns the following informally stated  general 
problem, which we study in the context of 
formal language identification and algorithmic learning theory:
\begin{equation}\label{v9CcOK3DX6}
\parbox{12cm}{
Given a probability distribution $\PP$ and a sufficiently large sample of
randomly chosen
data from the given distribution, learn or estimate
a probability distribution with respect to which the sample has been randomly sampled.
}
\end{equation}
Problem \eqref{v9CcOK3DX6} has a long history in statistics (\eg see \citep{Vapnik:1982:EDB})
and has more recently been approached in the context of computational learning,  
in particular the probably approximately correct (PAC) learning model, starting with
\citep{Kearns:1994:LDD}. The same problem was recently approached in the context of
 Algorithmic Learning Theory, in the tradition of \citep{GOLD1967447},
 and Kolmogorov complexity
by  \citep{VITANYI201713}.\footnote{Probabilistic methods and 
learning concepts in formal language and algorithmic learning theory 
have been studied long before \citep{VITANYI201713},
see \citep{Pitt:1989:PII} and the survey \citep{Ambainis:2001}. However most of this work focuses
on identifying classes of languages or functions using probabilistic strategies, rather than
identifying probability distributions as Problem \eqref{v9CcOK3DX6} asks.
Bienvenu and Monin \citep[Section IV]{Bienvenu:2012:VNB:235} do study a form of 
\eqref{v9CcOK3DX6} through a concept that they call {\em layerwise learnability} of 
probability measures
in the Cantor space,
but this is considerably different than 
 \citep{VITANYI201713} and the concepts of
Gold \citep{GOLD1967447}, the most important difference being that
it refers to classes of probability 
measures that are not necessarily contained in the computable
probability measures.}

The learning concepts discussed in   \citep{VITANYI201713} are
very similar in nature to the classic concepts of 
algorithmic learning
which are
motivated by the problem of language learning in the limit (see \citep{PINKER1979217})
but they differ in two major ways. 
In the classic setting, one starts with a class of languages or functions which have a finite
description (\eg they are computable) and the problem is to find an algorithm (often called
{\em a learner}) which can infer,
given a sufficiently long text from any language in the given class, or a sufficiently long
segment of the characteristic sequence of any function in the given class, a description
of the language or function in the form of a grammar  
or a program. More precisely, the desired algorithm makes
successive predictions given longer and longer segments of the input sequence, and is required
to converge to a correct grammar or program for the given infinite input.

If we apply the concept of identification in the limit to Problem  \eqref{v9CcOK3DX6}, according to
\citep{VITANYI201713}, one starts with a class $\VV$ of finitely describable
probability distributions (say,  the computable measures on the Cantor space) and we have the following differences with respect to the classic setting:
\begin{itemize}
\item the inputs on which the learner is supposed to succeed in the limit are random
sequences with respect to some probability distribution in the given class $\VV$,
and not elements of $\VV$;
\item success of the learner $\LL$ on input $X$ 
means that $\LL(X\restr_n)$ 
converges, as $n\to\infty$, to
a description of some element of $\VV$ with respect to which $X$ is random.
\end{itemize}
First, note that just as in the context of computational learning theory, here too
we need to restrict the probability distributions in Problem \eqref{v9CcOK3DX6}
to a class of `feasible' distributions, which in our case means computable distributions 
in the Cantor space. Second, in order to specify the learning concept we have described,
we need to define what we mean by random inputs $X$ with respect to a computable 
distribution $\PP$ in
the given class $\VV$ on which the learner is asked to succeed. 
\citep{VITANYI201713} ask the learner to succeed on every real $X$
which is {\em algorithmically random}, in the sense of \ml \citep{MR0223179}, with respect to some computable probability measure.\footnote{From this point on we will use the term 
(probability) {\em measure} instead of distribution, since the literature in algorithmic randomness
that we are going to use is mostly written in this terminology.} Then the interpretation of 
Problem  \eqref{v9CcOK3DX6} through the lenses of algorithmic learning theory
and in particular, the ideas of  \citep{VITANYI201713} is as follows:
\begin{equation}\label{sGTC9dqWCS}
\parbox{12cm}{Given a computable measure $\mu$ and an
algorithmically random stream $X$ with respect to $\mu$,
learn in the limit (by reading the initial segments of $X$) a
computable measure $\mu'$ with respect to which $X$ is
algorithmically random.}
\end{equation}

This formulation invites many different formalizations of learning concepts
which are parallel to the classic theory of algorithmic learning\footnote{EX-learning, BC-learning,
BC$^{\ast}$-learning etc. In \citep[Chapter VII.5]{Odifreddi:99} the reader can find
a concise and accessible introduction to these basic learning concepts and results.},
and although we will comment on some of them later on, this article is
specifically concerned with EX-learning ({\em explanatory learning}, one of the main
concepts in Gold \citep{GOLD1967447}), which means
that in \eqref{sGTC9dqWCS} we require the learner to eventually converge to a specific
description of the computable measure\footnote{as opposed to, for example,
eventually giving different indices of the same  
measure, or even different measures all of which satisfy the required properties.}
 with the required properties.

A {\em learner} is simply a function $\LL:\twomel\to\Nat$.
We refer to infinite binary streams (sequences) as reals.
 According to  \citep{GOLD1967447},
a class $\CC$ 
of computable elements of $\twome$ is EX-learnable if there exists a learner $\LL$
such that for each $Z\in\CC$ we have that $\lim_n \LL(Z\restr_n)$ exists and equals
an index of $Z$ as a computable function.\footnote{Here we identify subsets of natural numbers with their
characteristic function and assume a fixed G\"{o}del numbering 
$(\varphi_e)$
of all partial computable functions with binary values (also called {\em reals}), 
which gives an `index' to member of this class.} Similarly, $\CC$ is BC-learnable if
there exists a learner $\LL$ such that  for each $Z\in\CC$
there exists some $n_0$ such that for all $n>n_0$ the value of $\LL(Z\restr_n)$
is an index of $Z$.

In this paper we study explanatory (EX) learning, behaviorally correct (BC) learning 
and partial learning of probability
measures, based on the classic notion of algorithmic randomness 
by \citep{MR0223179}.
Given a measure $\mu$ on the reals and a real $X$, we say that $X$
is $\mu$-random if it is algorithmically random with respect to $\mu$.
We review 
algorithmic randomness with respect to
arbitrary measures in Section \ref{TTprzlKboB}.
\begin{defi}[EX learning of measures]\label{nAemubaidf}
A class $\CC$ of computable measures is EX-learnable if there exists
a computable learner $\LL:\twomel\to\Nat$ such that
for every $\mu\in\CC$ and every $\mu$-random real $X$ the limit
$\lim_n \LL(X\restr_n)$ exists and equals an index of a measure $\mu'\in\CC$ 
such that $X$ is $\mu'$-random.
\end{defi}
\citep{VITANYI201713} 
introduced this notion and 
observed that any uniformly computable family of measures is
EX-learnable. On the other hand, 
 \citep{Bienvenu2014}
showed that the class of computable measures is not EX-learnable,
and also not even BC-learnable in the following sense.


\begin{defi}[BC learning of measures]\label{1PuUYa3Jz4}
A class $\CC$ of computable measures is BC-learnable if there exists
a computable learner $\LL:\twomel\to\Nat$ such that
for every $\mu\in\CC$ and every $\mu$-random real $X$ 
there exists $n_0$ and $\mu'\in\CC$ such that for all $n>n_0$ the value
$\LL(X\restr_n)$ is an index of $\mu'$
such that $X$ is $\mu'$-random.
\end{defi}


One could consider a stronger learnability condition, namely that given $\mu\in\CC$ and
any $\mu$-random $X$ the learner identifies $\mu$ in the limit, when reading initial segments of $X$.
Note that such a property would only be realizable in classes $\CC$ where any $\mu,\mu'\in \CC$
are {\em effectively orthogonal}, which means that the classes of $\mu$-random and $\mu'$-random
reals are disjoint. In this case we call $\CC$ {\em effectively orthogonal}, and
Definitions \ref{nAemubaidf}, \ref{1PuUYa3Jz4} are equivalent with the versions where
$\mu'$ is replaced by $\mu$.
On the other hand we could consider a weakened notion of learning of a class $\CC$ of
computable measures, where 
given $\mu\in\CC$ and
any $\mu$-random $X$, the learner identifies {\em some}  computable measure $\mu$ 
(possibly not in $\CC$) in the limit, with respect to which $X$ is random, 
when reading initial segments of $X$.

\begin{defi}[Weak EX learning of measures]\label{Vr7RSPbOiE}
A class $\CC$ of computable measures is weakly EX-learnable if there exists
a computable learner $\LL:\twomel\to\Nat$ such that
for every $\mu\in\CC$ and every $\mu$-random real $X$ the limit
$\lim_n \LL(X\restr_n)$ exists and equals an index of a computable measure $\mu'$ 
such that $X$ is $\mu'$-random.
\end{defi}
\begin{defi}[Weak BC learning of measures]\label{MxTYFrcADu}
A class $\CC$ of computable measures is weakly BC-learnable if there exists
a computable learner $\LL:\twomel\to\Nat$ such that
for every $\mu\in\CC$ and every $\mu$-random real $X$ 
there exists $n_0$ and a computable measure $\mu'$, 
such that for all $n>n_0$ the value
$\LL(X\restr_n)$ is an index of $\mu'$
such that $X$ is $\mu'$-random.
\end{defi}
%
We note that 
the notions in Definitions
 \ref{nAemubaidf} and \ref{1PuUYa3Jz4} are not closed under subsets.
 In the following proof and the rest of this article, we use `$\ast$' to denote concatenation of strings.

\begin{prop}
There exist classes $\CC\subseteq\DD$ of measures such that $\DD$ is EX-learnable
 and $\CC$ is not even BC-learnable. 
\end{prop}
\begin{proof}
Let
$(\sigma_i)$ be a \pf sequence of strings, let 
$\mu_i$ be the measure with 
$\mu_i(\sigma_{2i}\ast 0^{\omega})=\mu_i(\sigma_{2i+1}\ast 0^{\omega})=1/2$
and let
$\nu_i$ be the measure such that $\nu_i(\sigma_{i}\ast 0^{\omega})=1$.
Define $\CC=\{\mu_i, \nu_{2j}\ |\ i\in \emptyset'''\ \wedge j\in \Nat - \emptyset'''\}$ and
$\DD=\{\mu_i,\nu_i\ |\  i\in\Nat\}$. Clearly $\CC\subseteq \DD$.
If $\CC$ was BC-learnable then $\emptyset'''$ could be decided in $\emptyset''$:
to decide if $n\in \emptyset'''$ we just need to 
check the limit guess of the learner on $\sigma_{2n}$.  This is a contradiction.
On the other hand the learner 
which guesses $\nu_i$ on each
extension of $\sigma_i$ is an EX-learner for $\DD$.
\end{proof}

 On the other hand, the weaker notions of 
 Definitions \ref{Vr7RSPbOiE} and \ref{MxTYFrcADu} clearly are closed under subsets.
 %
%
In Section \ref{PC4LgRBwRx} we also consider an analogue of
the notion of partial learning from \citep{STL1st} for measures, and prove an analogue
of the classic result from the same book that the computable reals are partially learnable.

\subsection{Our main results}\label{PC4LgRBwRx}
The aim of this paper is to establish a connection between the above notions of learnability
of probability measures, with the corresponding classical notions of learnability of reals
in the sense of  \citep{GOLD1967447}. To this end, we prove the following equivalence theorem,
which allows to transfer positive and negative learnability results from reals to probability
measures that are parametrized by reals, and vice-versa. 
Let $\MM$ denote the Borel measures on $\twome$.\footnote{Formal background 
definitions regarding the metric space of Borel measures are given in Section \ref{M7XSECbkKz}.}

\begin{thm}[The first equivalence theorem]\label{d3lcvKUJi}
Given a computable 
$f:\twome\to\MM$ 
let $\DD\subseteq \twome$ be an effectively closed set 
such that 
for any $X\neq Y$ in $\DD$
the measures $f(X), f(Y)$ are  effectively orthogonal. 
If $\DDast\subseteq\DD$ is a class of computable reals,  
$\DDast$ is EX-learnable if and only if $f(\DDast)$ is EX-learnable. 
The same is true of the BC learnability of $\DDast$.
\end{thm}
As a useful and typical 
example of a parametrization $f$ of measures by reals as stated in Theorem
\ref{d3lcvKUJi}, consider the function that maps each real $X\in \twome$
to the Bernoulli measure with success probability the real in the unit interval $[0,1]$
with binary expansion $X$. Note that the Bernoulli measures\footnote{By Bernoulli measure we mean the product
measure on the space of infinite binary strings, of the `biased coin' measure on $\{0,1\}$ 
that gives probability $q\in[0,1]$ on 0 and probability $1-q$ on 1.} are 
an effectively orthogonal class
(\eg consider the law of large numbers regarding the frequency of 0s in the limit).
The proof of Theorem \ref{d3lcvKUJi} is given in 
Section \ref{XS3k8tMXTU}. We note that It is possible to relax the hypothesis of the
`if' direction of Theorem \ref{d3lcvKUJi} for the case of 
EX-learning -- we give this extension in Section \ref{F8FbxE437v}.

The next equivalence theorem concerns weak learnability.\footnote{For 
the special case where we allow measures with atoms in our classes, 
Theorem \ref{WVmrKGDjbs} has
a somewhat easier proof than the one given in Section \ref{In6hhYHW}.}
\begin{thm}[The second equivalence theorem]\label{WVmrKGDjbs}
There exists a map $Z\to\mu_Z$ from $\twome$ to the continuous Borel measures 
on $\twome$, such that
for every class $\CC$ of computable reals, $\CC$ is EX/BC learnable
if and only if $\{\mu_Z\ |\ Z\in\CC\}$ is a weakly EX/BC learnable
class of computable measures, respectively.
\end{thm}
%

Finally we give a positive result in terms of partial learning.
Let $(\mu_e)$ be a uniform enumeration of all partial computable measures (see Section~\ref{Zs7iCd895A}).
We say that a learner $\LL$ {\em partially succeeds} on a computable measure $\mu$ if
for all $\mu$-random $X$ there exists a $j_0$ such that
(a) there are infinitely many $n$ with $\LL(X\restr_n)=j_0$;
(b) if $j\neq j_0$ then there are only finitely  many $n$ with $\LL(X\restr_n)=j$;
(c) $\mu_{j_0}$ is a computable measure such that $X$ is $\mu_{j_0}$-random.
\begin{thm}\label{EHY3duonav}
There exists a computable learner which partially succeeds on all computable measures.
\end{thm}
Theorems 
\ref{d3lcvKUJi} and \ref{WVmrKGDjbs} allow 
the transfer of learnability results from the classical theory
on the reals to probability measures.
Detailed background on the 
notions that are used in our results and their proofs is given in
Section \ref{M7XSECbkKz}.

\subsection{Applications of our main results}\label{Qx1nj3CucT}
The equivalences
in Theorems \ref{d3lcvKUJi} and \ref{WVmrKGDjbs}
have some interesting applications, some of which are
stated below, deferring their proofs to Section \ref{P6hQ2fd9BA}.
%
%

 \citep{AdlemanB91} showed that an oracle can EX-learn all
computable reals if and only if it is {\em high}, i.e.\ it computes a function that
dominates all computable functions. Using Theorem \ref{WVmrKGDjbs} we may
obtain the following analogue for measures.
\begin{coro}\label{gyGJY8fGL}
The computable (continuous) measures are (weakly) EX-learnable with oracle $A$
if and only if $A$ is high.
\end{coro}
We may write EX$[A]$ to indicate that the
EX-learner is computable in $A$.
A class $\CC$ of measures is (weakly) EX$\ast[A]$-learnable for an oracle $A$, if
there exists an EX-learner $\LL\leq_T A$ for $\CC$ such that for each $X$,
the function $n\to \LL(X\restr_n)$
uses finitely many queries to $A$.
The following is an analogue of a result from \citep{KUMMER1996214}
about EX$\ast[A]$ learning of reals.
\begin{coro}\label{UT2HmCWUPe}
The class of computable measures 
is EX$\ast[A]$-learnable
if and only if $\emptyset''\leq_T A\oplus \emptyset'$.
\end{coro}

If we apply Theorem \ref{d3lcvKUJi} we
obtain an analogue of the \citep{AdlemanB91} characterization 
with respect to Bernoulli measures.
\begin{coro}\label{P2b9kifIJR}
An oracle can EX-learn all computable Bernoulli measures if and only if it is high.
\end{coro}
%
\citep{BLUM1975125} showed the so-called non-union theorem
for EX-learning, namely that EX-learnability of classes of computable reals is not closed under union.
We may apply our equivalence theorem in order to prove an analogue for measures.
\begin{coro}[Non-union  for measures]\label{A3ce2GoRBo}
There are two EX-learnable classes of computable (Bernoulli) 
measures such that their union is not EX-learnable.
\end{coro}
One can find applications of Theorem \ref{d3lcvKUJi}
on various more complex results in algorithmic learning theory.
As an example, we mention the characterization of low oracles for EX-learning that
was obtained in \citep{Gasarch1989214, SlamanSol:1991} (also see \citep{FORTNOW1994231}).
An oracle $A$ is low for EX-learning of classes of computable measures, if any
class of computable measures that is learnable with oracle $A$, is learnable without any oracle.
The characterization mentioned above is that, an oracle is low for EX-learning if and only if
it is 1-generic and computable from the halting problem.
This argument consisted of three steps, first showing that 1-generic oracles computable from the halting problem
are low for EX-learning, then that oracles that are not computable from the halting problem are not low for EX-learning,
and finally that oracles that are computable from the halting problem but are not 1-generic
are not low for EX-learning. The last two results can be combined with Theorem \ref{d3lcvKUJi} in order to
show one direction of the characterization for measures:
\begin{equation}\label{maMViSH2Qa}
\parbox{13cm}{if an oracle $A$ is either not computable from the halting problem or
not 1-generic, then there exists a class of computable (Bernoulli) measures which is 
not EX-learnable but which is EX-learnable with oracle $A$.}
\end{equation}
In other words,
%
%
low for EX-learning oracles for measures
are 1-generic and computable from the halting problem.
\begin{coro}\label{ounvKl7AOx}
If an oracle is low for EX-learning for measures, then it is also
low for EX-learning for reals.
\end{coro}
We do not know if the converse of Corollary \ref{ounvKl7AOx} holds.

\subsection{Notions of learnability of probability measures}\label{mp7nBBNqQm}

\citep{Bienvenu2014}
say that a learner $\LL$ EX-succeeds on a real $X$ if
$\lim_n \LL(X\restr_n)$ equals an index of a computable 
measure with respect to which $X$ is random.
Similarly, $\LL$ BC-succeeds on $X$ if there exists a measure $\mu$
such that $X$ is $\mu$-random, and for all sufficiently large $n$,
the value of $\LL(X\restr_n)$ is an index of $\mu$.
The results in \citep{Bienvenu2014,Bienvealgindpro}  are of the form `there exists (or not)
a learner which succeeds on all reals that are random with respect to
a computable measure'. Hence \citep{Bienvenu2014,Bienvealgindpro} 
refer to the weak learnability of Definitions \ref{Vr7RSPbOiE}
and \ref{MxTYFrcADu}.

\citep{Bienvenu:2012:VNB:235} introduced and studied
{\em layerwise learnability},
in relation to uniform randomness extraction from biased coins.
This notion is quite different from learnability in the sense of algorithmic learning theory,
but it relates to the `only if' direction of Theorem \ref{d3lcvKUJi}.
Let $\MM$ denote the class of Borel measures on $\twome$.\footnote{Refer to Section
\ref{M7XSECbkKz} for background on the notions used in this discussion.}
A class $\CC\subseteq\MM$ of measures (not necessarily computable) is layerwise learnable
if there is a computable function $F:\twome\times\Nat\to\MM$ which,
given any $\mu\in\CC$ and any $\mu$-random real $X$, if the
$\mu$-randomness deficiency of $X$ is less than $c$ then
$F(X,c)=\mu$. In other words, this notion of learnability of a class $\CC\subseteq\MM$ requires
to be able to compute (as an infinite object) any measure $\mu\in\CC$ from any
$\mu$-random real and a guarantee on the level of $\mu$-randomness of the real. 
Hence
the main difference with the notions in
Definitions  \ref{nAemubaidf} and \ref{1PuUYa3Jz4}
 is that (a) we also learn
incomputable measures and (b) learning does not identify a finite program describing the measure,
but it computes a measure as an oracle Turing machine infinite computation with oracle the
random real. 
As a concrete example of the difference between the two notions, consider the
class of the computable Bernoulli measures which is layerwise learnable 
 \citep{Bienvenu:2012:VNB:235} but is not (weakly) EX-learnable 
or even (weakly) BC-learnable by \citep{Bienvenu2014}.
 
\section{Background}\label{M7XSECbkKz}
We briefly review the background on  the Cantor space $\twome$ and
the space of Borel measures that is directly relevant 
for understanding our results and proofs. We focus on effectivity properties
of these concepts and the notion of algorithmic randomness. This is textbook material
in computable analysis and algorithmic randomness, and we have chosen
a small number of references where the reader can obtain more detailed presentations
that are similar in the way we use the notions here.

\subsection{Representations of Borel measures on the Cantor space}\label{Zs7iCd895A}
We view $\twome$ and the space $\MM$ of
Borel measures on $\twome$ as computable metric spaces.\footnote{All of the notions and facts discussed in
this section are standard in computable analysis and are presented in more detail in
\citep{Bienvenu:2012:VNB:235,Bienvealgindpro}. More general related facts, such that
the fact that for any computable metric space $\CC$ 
the set of probability measures over $\CC$ is itself a computable metric space,
can be found in \citep{Gacs:2005:UTA}.}
The distance between two
reals is $2^{-n}$ where $n$ is the first digit where they differ, and the basic open sets
are $\dbra{\sigma}=:\{X\in\twome\ |\ \sigma\preceq X\}$, $\sigma\in \twomel$,
where $\preceq$ denotes the prefix relation. If $V\subseteq\twomel$ then
$\dbra{V}:=\cup_{\sigma\in V} \dbra{\sigma}$.
The distance between $\mu,\nu\in\MM$ is given by
\[
d(\mu,\nu)=\sum_n 2^{-n}\cdot\Big(\max_{\sigma\in 2^n}  |\mu(\sigma)-\nu(\sigma)|\Big)
\]
The basic open sets of $\MM$ are the balls of the form
\[
\big[(\sigma_0,I_0),\dots, (\sigma_n,I_n)\big]=
\big\{\mu\in\MM\ |\ \forall i\leq n,\ \mu(\sigma_i)\in I_i\big\}
\]
where $\sigma_i$ are binary strings (which we identify with the open balls 
$\dbra{\sigma}$ of
$\twome$) and $I_i$ are the {\em basic open intervals in} $[0,1]$.\footnote{These are the
intervals $(q,p), [0,q), (p,1]$ for all dyadic rationals $p,q\in (0,1)$.}
Define the size of a basic open set $C$ of $\MM$ by
\[
\big| C \big|= \sup \{d(\mu,\nu)\ |\ \mu,\nu\in C\},
\hspace{0.2cm}\textrm{for $C\in\MMast$}
\]
and note that this is a computable function.
By the Caratheodory theorem, 
each $\mu\in\MM$ is uniquely determined by its values on the basic open sets
of $\twome$, namely the values $\mu(\sigma):=\mu(\dbra{\sigma})$, $\sigma\in\twomel$.
Also,  each $\mu\in\MM$ is uniquely determined by the basic open sets that contain it, 
and the same is true for $\twome$.
A subset of $\MM$ is effectively open if it is the union of a computably enumerable
set of basic open sets.

We represent measures in $\MM$ as the functions 
$\mu: \twomel\to [0,1]$ such that $\mu(\emptyset)=1$ (here $\emptyset$ is the empty string)
and
$\mu(\sigma)=\mu(\sigma\ast 0)+\mu(\sigma\ast 1)$ for each $\sigma\in\twomel$.
We often identify a measure with its representation.
A measure $\mu$ is computable if its representation is computable as a
real function.
There are two equivalent ways to define what an index (or description)
of a computable measure is. One is to define it as a computable approximation
to it with uniform modulus of convergence. For example,
we could say that  
a {\em partial computable measure} is
a \ce (abbreviation for `computably enumerable') 
set $W$ of basic open sets $(\sigma, I)$ of $\MM$, where $\sigma\in\twomel$, 
$I$ is a basic open interval
of $[0,1]$;\footnote{If one wishes to ensure that in case of convergence
the property $\mu(\sigma)=\mu(\sigma\ast 0)+\mu(\sigma\ast 1)$ holds,
we could also require that
if $(\sigma, I),(\sigma\ast 0, J_0),(\sigma\ast 1, J_1)\in W$ then 
$I\cap [\inf J_0 +\inf J_1,\sup J_0 +\sup J_1]\neq\emptyset$.}
In this case we can have a uniform enumeration $(\mu_e)$ of all partial computable measures,
which could contain non-convergent approximations. Then $\mu_e$, represented by the \ce set 
$W_e$,  is total and equal to 
some measure $\mu$ if
$\mu\in [(\sigma, I)]$ for all $\sigma,I$ with $(\sigma, I)\in W_e$, and for each $\sigma$ we have
$\inf\{|I|\ |\ (\sigma,I)\in W_e\}=0$.
Alternatively, one could consider the fact that for every computable measure $\mu$
there exists a computable measure $\nu$ which takes dyadic values on each string $\sigma$,
and such that $\mu=\Theta(\nu)$ (\ie the two measures are the same up to a multiplicative
constant -- see 
\citep{JLweak}).
Moreover, from $\mu$ one can effectively define $\nu$, and the property $\mu=\Theta(\nu)$
implies that the $\mu$-random reals are the same as the $\nu$-random reals.
This means that we may restrict our considerations to 
the computable measures with dyadic rational values on every string, without loss of generality.
Then we can simply let $(\mu_e)$ be an effective list of all partial computable
functions from $\twomel$ to the dyadic rationals such that 
$\mu(\sigma)=\mu(\sigma\ast 0)+\mu(\sigma\ast 1)$ for each $\sigma$ such that
the values $\mu(\sigma),\mu(\sigma\ast 0),\mu(\sigma\ast 1)$ are defined.

The two formulations are effectively equivalent, in the sense that
from one we can effectively obtain the other, so we do not explicitly distinguish them.
In any case, an index of a computable measure $\mu$ is a number $e$ such that
$\mu_e$ is total and equals $\mu$. An important exception to this equivalence is when
we consider subclasses of computable measures, such as the computable Bernoulli measures
which feature in Section \ref{P6hQ2fd9BA}. In this case we have to use the first definition of
$(\mu_e)$ above, since it is no longer true that every computable Bernoulli measure can be
replaced with a computable Bernoulli measure with dyadic values which has the same 
random reals.

\subsection{Computable functions and metric spaces}
There is a well-established notion of a computable function $f$ 
between computable metric spaces
from computable analysis, \eg see \citep{Bienvenu:2012:VNB:235,WEIHRAUCH1993191}.
The essence of this notion is effective continuity, \ie 
that for each $x$ and a prescribed error bound $\epsilon$ for
an  approximation to $f(x)$, one can compute a neighborhood radius around 
$x$ such that all of the $y$ in the neighborhood are mapped within distance $\epsilon$ from
$f(x)$. Here we only need the notion of a computable function $f:\twome\to\MM$, which
can be seen to be equivalent to the following (due to the compactness of $\twome$).
Let $\MMast$ denote the collection of the basic open sets of $\MM$.
\begin{defi}\label{9hsw3pE4Wp}
A function $f:\twome\to\MM$ is computable if there exists a computable
function $f^{\ast}:\twomel\to\MMast$ which is monotone in the sense that
$\sigma\preceq\tau$ implies 
$f^{\ast}(\sigma)\subseteq f^{\ast}(\tau)$, and
such that for all $Z\in\twome$ we have $f(Z)\in f^{\ast}(Z\restr_n)$ for all $n$,
and $\lim_s |f^{\ast}(Z\restr_s)|=0$.
\end{defi}
More generally, a computable metric space is a tuple $(\XX, d_x, (q_i))$
such that $(\XX,d_x)$ is a complete separable metric space,
$(q_i)$ is a countable dense subset of $\XX$
and the function $(i,j)\mapsto d_x(q_i,q_j)$ is computable.
A function $f:\XX\to\YY$ between two computable metric spaces
$(\XX, d_x, (q^x_i))$, $(\YY, d_y, (q^y_i))$
is computable if there exists computable function
$g$ such that
for every  $n\in\Nat$
and every $w,z\in\XX$ such that $d_x(w,z)<2^{-g(n)}$ we have
$d_y(f(w),f(z))< 2^{-n}$; equivalently, if
for all $n,i,j\in\Nat$, such that
$d_x(q^x_i,q^x_j)<2^{-g(n)}$ we have
$d_y(f(q^x_i),f(q^x_j))< 2^{-n}$.
In this way, as it is illustrated in
Definition \ref{9hsw3pE4Wp},
computable functions between $\twome,\Nat,\MM$ and their
products can be thought of as induced by  monotone computable functions
between the corresponding classes of basic open sets, 
such that the sizes of the images decrease
uniformly as a function of the size of the arguments.

\subsection{Algorithmic randomness with respect to arbitrary measures}\label{TTprzlKboB}
There is a robust notion of algorithmic randomness with respect to an arbitrary
measure $\mu$ on $\twome$, which was manifested in approaches by 
\citep{levinuniftests,leviniandc/Levin84} and 
\citep{Gacs:2005:UTA}
in terms of uniform tests, and in \citep{ReiSlaMRR} 
in terms of representations of measures,
all of which were shown to be equivalent by  \citep{DMncompmea}.
In this paper we will mainly use the specific case when the measure is computable,
which is part of the classic definition of  \citep{MR0223179}.
Given a computable measure $\mu$, a \ml $\mu$-test is a 
uniformly \ce sequence $(U_i)$ of sets of strings (viewed as the sets
of reals with prefixes the strings in the sets) such that $\mu(U_i)<2^{-i}$ for each $i$.
A real $Z$ is $\mu$-random if it is not contained in  the intersection of any \ml $\mu$-test.
By \citep{MR0223179} there exists a universal \ml $\mu$-test (uniformly in $\mu$)
\ie a  \ml $\mu$-test $(U_i)$ with the property that the set of $\mu$-random reals 
is $\twome -\cap_i \dbra{U_i}$ (where $\dbra{U_i}$ denotes the set of reals with
prefixes in $U_i$).
Equivalently, if $K$ is the \pf Kolmogorov complexity function,
$Z$ is $\mu$-random if there exists $c\in\Nat$ 
such that $\forall n\ K(Z\restr_n)> -\log \mu(Z\restr_n)-c$.
Occasionally it is useful to refer to the {\em randomness deficiency} of a real,
which can be defined in many equivalent ways.\footnote{Equivalent\label{v9SlhIOuiN}
in the sense that from an upper bound of one notion with respect to a real, 
we can effectively obtain an upper bound on another notion with respect to the same real.}
For example, we could define $\mu$-deficiency 
to be the least $i$ such that $Z\not\in\dbra{U_i}$
where $(U_i)$ is the universal \ml $\mu$-test, or 
$\sup_n (\ceil{-\log \mu(Z\restr_n)}-K(Z\restr_n))$.
Clearly $Z$ is $\mu$-random if and only if it has finite $\mu$-deficiency.
Randomness with respect to arbitrary measures only plays a role in Section \ref{IvylEyj89U}.
We define it in terms of randomness deficiency, following \citep{Bienvealgindpro}.
We define the (uniform) {\em randomness deficiency function}
to be the largest, up to an additive constant, 
function
$\dd:\twome\times\MM\to\Nat\cup\{\infty\}$ 
such that\footnote{We can get a precise definition of $\dd$ by starting
with a universal enumeration $W_e(k)$ all uniform c.e. sequences of sets $W(k)$, 
where each $W(k)$ is a set of pairs $(\sigma, I)$ of basic open sets of
$\twome,\MM$ respectively (viewed as basic open set of the product space
$\twome\times\MM$) 
with the property that for each $\mu \in I$ 
$\mu(\dbra{\{\sigma\ |\ (\sigma, I) \in W(k)\}})<2^{-k}$.
Then define $\dd(X,\mu)= \sum_e 2^{-e} \cdot w_e(X,\mu)$
where $w_e(X,\mu)$ is the maximum $k$ such that $(X,\mu)$ is 
in the open set $W_e(k)$.}
\begin{itemize}
\item the sets $\dd^{-1}((k,\infty))$ are effectively open uniformly in $k$;
\item $\mu(\{X\ |\ \dd(X,\mu)>k\})<2^{-k}$,
$\forall X\in\twome$, $\mu\in\MM$, $k\in\Nat$.
\end{itemize}
Given any $\mu\in\MM$ and $Z\in\twome$, the $\mu$-deficiency of $Z$ is
$\dd(Z,\mu)$ and $Z$
is $\mu$-random if it has finite $\mu$-deficiency.
This definition is based on the uniform tests approach as mentioned before,
and is equivalent to \ml randomness for computable measures. Moreover,
the deficiency notions are equivalent in the sense of footnote \ref{v9SlhIOuiN}.
The reader may find additional background on algorithmic randomness in
the monographs \citep{Li.Vitanyi:93} and \citep{rodenisbook}.

\section{Proof of Theorem \ref{d3lcvKUJi} and Theorem \ref{WVmrKGDjbs}}\label{XS3k8tMXTU}
We start with Theorem \ref{d3lcvKUJi}.
Let $\DD\subseteq \twome$ be an effectively closed set 
and let $\DDast\subseteq\DD$ contain only computable reals.
Also let $f:\twome\to\MM$ 
be a computable function
such that  for any $X\neq Y$ in $\DD$
the measures $f(X), f(Y)$ are  effectively orthogonal. 
The easiest direction of 
 Theorem \ref{d3lcvKUJi} is that if
$\DDast$ is (EX or BC) learnable then $f(\DDast)$ is (EX or BC, respectively) learnable,
and is proved in Section \ref{IvylEyj89U}. We 
stress that the effective orthogonality property of $f$, and hence the fact that
it is injective, is used in a crucial way  in the argument of 
Section \ref{IvylEyj89U}.
Sections \ref{g2uXHFLpP} and \ref{EZtHGpioBE}
prove the `if' direction of Theorem \ref{d3lcvKUJi} for EX and BC learnability
respectively, and are the more involved part of this paper.
In Section \ref{In6hhYHW} we prove  Theorem \ref{WVmrKGDjbs}.

\subsection{From learning reals to learning measures}\label{IvylEyj89U}
We show the `only if' direction of Theorem \ref{d3lcvKUJi},
first for EX learning and then for BC learning.
Let $f,\DD,\DDast$ be as in the statement of Theorem \ref{d3lcvKUJi}.
Since $f$ maps distinct reals in $\DD$ to effectively orthogonal measures, 
given $X\in\twome$ there exists at most one
$\mu\in f(\DD)$ such that $X$ is $\mu$-random. By the properties of $f$, there is also at most
one $Z\in \DD$ such that $X$ is $f(Z)$-random.
It follows from the definition of deficiency in Section \ref{TTprzlKboB}, that
for each $X\in\twome$, $c\in\Nat$,  the class of 
$Z\in \DD$ such that $X$ is $f(Z)$-random with deficiency $\leq c$ is a 
$\Pi^0_1(X)$ class $P(X,c)$ (uniformly in $X,c$).
By the effective orthogonality of 
the image of $\DD$ under $f$, the latter class either contains a unique real, or is empty.
Moreover, the latter case occurs if and only if
there is no $\mu\in f(\DD)$ with respect to which $X$ is $\mu$-random
with deficiency $\leq c$.
Now note that given a
$\Pi^0_1(X)$ class $P \subseteq\twome$, by compactness the emptiness of $P$ is
a $\Sigma^0_1(X)$ event, and if $P$ contains a unique path, this path is uniformly computable
from $X$ and an index of $P$.

It follows that there exists a computable function $h: \twomel\to\twomel$ 
such that
for all $X$ which is $f(Z)$-random for some $Z\in \DD$, 
\begin{itemize}
\item $\lim_s |h(X\restr_s)|=\infty$;
\item there exists $n_0$ such that for all $m>n>n_0$ we have
$h(X\restr_n)\preceq h(X\restr_m)$;
\item as $n\to\infty$ the prefixes  $h(X\restr_n)$ converge to 
the unique real $Z\in \DD$ such that $X$ is $f(Z)$-random.
\end{itemize}
Indeed, on the initial segments of $X$, the function $h$ will start
generating the classes $P(X,c)$ as we described above, starting with $c=0$
and increasing $c$ by 1 each time that the class at hand becomes empty.
While this process is fixed on some value of $c$, it starts producing the initial segments
of the unique path of $P(X,c)$ (if there are more than one path, this process will 
stop producing longer and longer strings, reaching a finite partial limit). In the special case
that $X$ is $f(Z)$-random for some $Z\in\DD$, such a real $Z\in \DD$ is unique,
and the process will reach a limit value of $c$, at which point it will produce
a monotone sequence of longer and longer prefixes of $Z$. \footnote{Alternatively, in order to
obtain $h$, one can make use of a result from 
\citep{Bienvenu:2012:VNB:235}.
Since
$f$ is computable, $\twome$ is compact and $\DD$ is effectively closed, the image $f(\DDast)$ is compact and effectively closed,
and the set of indices
of computably enumerable sets of basic open sets of $\MM$ whose union
contains $f(\DDast)$ is itself computably enumerable. In the terminology of
\citep{Bienvenu:2012:VNB:235}, the image 
$f(\DDast)$ is {\em effectively compact}.
Bienvenu and Monin \citep{Bienvenu:2012:VNB:235}
showed that if a class $\CC$ of effectively orthogonal measures is effectively
compact then
there exists a computable function $F:\twome\times\Nat\to\MM$ such that
$\forall \mu\in\CC\ \forall X\in\twome\ \forall c\in\Nat\ \ u(X,\mu)<c\Rightarrow F(X,c)=\mu$
where $u(X,\mu)$ is the $\mu$-deficiency of $X$.
One can derive the existence of $h$ from this result.}

Note that since $f:\twome\to\MM$ is computable, there exists
a computable $g:\Nat\to\Nat$ such that for each $e$, if $e$ is
an index of a computable $Z\in\twome$, then $g(e)$ is
an index of the computable measure $f(Z)$.

We are ready to define an EX-learner $\VV$  
for $f(\DDast)$, given an EX-learner $\LL$ for $\DDast$ and the functions $h,g$ that
we defined above.
For each $\sigma$ we let $\VV(\sigma)=g(\LL(h(\sigma)))$.
It remains to verify that for each $X$ which is $\mu$-random for some computable 
$\mu\in f(\DDast)$, the limit $\lim_s \VV(X\restr_s)$ exists and equals an index for 
(the unique such) $\mu$.
By the choice of $X$ and $h$ we have that there exists some $s_0$ such that for
all $s>s_0$, the string
$h(X\restr_s)$ is an initial segment of the unique $Z\in\DDast$ such that $f(Z)=\mu$;
moreover, $\lim_s |h(X\restr_s)|=\infty$ and since $\mu$ is computable and $\DD$ is effectively
closed, it follows that $Z$ is computable.
Hence, since $\LL$ learns all  reals in $\DDast$,
we get that $\lim_s \LL(h(X\restr_s))$ exists and is an index of $Z$. Then by the properties
of $g$ we get that $g(\lim_s \LL(h(X\restr_s)))=\lim_s g(\LL(h(X\restr_s)))$ is an index for $\mu$.
Hence $\lim_s \VV(X\restr_s)$ is an index of the unique computable $\mu\in f(\DDast)$
with respect to which $X$ is random, which concludes the proof.

Finally we can verify that the same argument shows that if $\DDast$ is BC-learnable, then
$f(\DDast)$ is BC-learnable. The definitions of $h,g$ remain the same. The only change is that
now we assume that $\LL$ is a BC-learner for $\DDast$.
We define the BC-learner $\VV$ for $f(\DDast)$ in the same way: 
$\VV(\sigma)=g(\LL(h(\sigma)))$.
As before, given $X$ such that there exists (a unique) $Z\in\DDast$
such that $X$ is $f(Z)$-random, we get that 
there exists some $s_0$ such that for
all $s>s_0$, the string
$h(X\restr_s)$ is an initial segment of the unique computable $Z\in\DDast$ such that $f(Z)=\mu$,
and moreover, $\lim_s |h(X\restr_s)|=\infty$.
Since $\LL$ is a BC-learner for $\DDast$, 
there exists some $s_1$ such that for all $s>s_1$
the integer $\LL(h(X\restr_s))$ is an index for the computable real $Z$.
Then by the properties of $g$ we get that for all
$s>s_1$, the integer $g(\LL(h(X\restr_s)))$ is an index for the computable measure $f(Z)$.
Since $X$ is $f(Z)$-random, this concludes the proof of the BC clause of
the `only if' direction of  Theorem \ref{d3lcvKUJi}.

\subsection{From learning measures to learning reals: the EX case}\label{g2uXHFLpP}
We show the `if' direction of the EX case of Theorem \ref{d3lcvKUJi}.
Let $f,\DD,\DDast$ be as given in the theorem and suppose that
$f(\DDast)$ is
EX-learnable.
This means that there exists a computable learner $\VV$ such that
for every $Z\in \DDast$ and every $f(Z)$-random $X$,
the limit $\lim_s \VV(X\restr_s)$ exists and is an index
of $f(Z)$. We are going to construct a learner $\LL$ for 
$\DDast$ so that
for each $Z\in \DDast$ the limit 
$\lim_s \LL(Z\restr_s)$ exists and is an index for $Z$.
Since $\DD$ is effectively closed and $f$ is computable and injective on $\DD$, 
by the compactness of $\twome$, 
\begin{equation}\label{xD7HzXSEpd}
\parbox{13cm}{there exists a computable $g:\Nat\to\Nat$ such that for each $e$,
if $e$ is an index of a computable $\mu\in f(\DD)$, the image $g(e)$ is an index
of the unique and computable $Z\in\DD$ such that $f(Z)=\mu$.}
\end{equation}
Hence it suffices to
\begin{equation}\label{oPZCcbP3K}
\parbox{13cm}{construct a computable function $\LLast:\twomel\to\Nat$
with the property that 
for each  $Z\in \DDast$ the limit 
$\lim_s \LLast(Z\restr_s)$ exists and is an index for $f(Z)$}
\end{equation}
because then the function $\LL(\sigma)=g(\LLast(\sigma))$ will be a computable
learner for $\DDast$.

Since $f:\twome\to\MM$ is computable, there exists a computable $f^{\ast}:\twomel\to\MMast$
(where $\MMast$ is the set of basic open sets of $\MM$) and a computable
increasing $h:\Nat\to\Nat$ such that:
\begin{itemize}
\item $\sigma\preceq\tau$ implies $f^{\ast}(\sigma) \subseteq f^{\ast}(\tau)$;
\item for all $Z\in\twome$, $\lim_s f^{\ast}(Z\restr_s)=f(Z)$;
\item for all $n$ and all $\sigma\in 2^{h(n)}$ the size of $f^{\ast}(\sigma)$ is at most $2^{-3n}$.
\end{itemize}
Note that by the properties of $f^{\ast}$ we have
\begin{equation}\label{S72AlHj9Wf}
\parbox{13cm}{for each $Z\in\twome$, each $n\in\Nat$ 
and any measures $\mu,\nu\in f^{\ast}(Z\restr_{h(n)})$ we have
$\sum_{\sigma\in 2^n} |\mu(\sigma)-\nu(\sigma)| < 2^{-n}$.}
\end{equation}
Below we will also use the fact that
\begin{equation}\label{LhxifOccwG}
\parbox{13cm}{there is a computable function that takes as input
any basic open interval $I$ of $\MM$ and returns (an index of) a computable
measure (say, as a measure representation) $\mu\in I$.}
\end{equation}

\textbf{Proof idea.}
Given $Z\in\DDast$ we have an approximation to the measure $\mu^{\ast}=f(Z)$.
Given $\mu^{\ast}$ and $\VV$ we get a majority vote on each of the levels
of the full binary tree, where each string $\sigma$ votes for the index $\VV(\sigma)$
and its vote has weight $\mu^{\ast}(\sigma)$. In search for the index of 
$Z\in\DDast$ we approximate the weights of the various indices
as described above, and aim to chose  an index with a positive weight.
If $\VV$ EX-learns $\mu^{\ast}$, it follows that such an index will indeed be
an index of $\mu^{\ast}$. One obvious way to look for such an index is at each
stage to choose the index whose current approximated weight is the largest.
This approach has the danger that there may be two different indices with the same weight,
in which case it is possible that the said approximation $\lim_n \LL^{\ast}(X\restr_n)$ does not
converge. We deal with this minor issue by requiring a sufficient difference on the
current weights for a change of guess.

\textbf{Construction of $\LLast$.}
We let $\LLast$ map the empty string to index 0 and for every other string $\sigma$ 
we define $\LLast(\sigma)$ as follows. If $\sigma\not\in \{2^{h(n)} \ |\ n\in\Nat\}$
 then let $\LLast(\sigma)=\LLast(\tau)$
where $\tau$ is the longest prefix of $\sigma$ in $\{2^{h(n)} \ |\ n\in\Nat\}$.
So it remains to define $\LLast$ in steps, where at step $n$
we define $\LLast$ on all strings $\sigma\in 2^{h(n)}$.
Since $f^{\ast}(\sigma)$ is basic open interval in $\MM$, we may use
\eqref{LhxifOccwG} in order to get a
computable function $\sigma\to\mu_{\sigma}$
from strings to computable measures, such that
for each $\sigma$ the measure $\mu_{\sigma}$ belongs to
$f^{\ast}(\sigma)$.\footnote{The reader should not confuse this notation with
the notation $(\mu_e)$ that we used for the universal list of all computable measures.}

Given $n$ and $\sigma\in 2^{h(n)}$, 
for each $e$ define 
\[
\wgt{e}=\mu_{\sigma}(\{\tau\in 2^{n}\ |\ \VV(\tau)=e\}).
\]
Let $e^{\ast}$ be the least number with the maximum $\wgt{e}$.\footnote{Note that there are 
at most $2^n$ many $e$ with $\wgt{e}\neq 0$ so this maximum exists. Moreover, we can
compute the set of these numbers $e$, the maximum and $e^{\ast}$, 
by computing $\mu_{\sigma}(\tau)$ and $\VV(\tau)$
for each $\tau\in 2^n$.}
Let $\sigma^-$ denote the first $|\sigma|-1$ bits of  $\sigma$.
If $\wgt{e^{\ast}}>3\cdot \wgt{\LLast(\sigma^-)}$,  
let $\LLast(\sigma)=e^{\ast}$; otherwise let $\LLast(\sigma)=\LLast(\sigma^-)$.

\textbf{Properties of $\LLast$.}
It remains to show \eqref{oPZCcbP3K}, so let $Z\in \DDast$.
First we show the claimed convergence and then that the limit is an index for $f(Z)$.
Let $\mu^{\ast}:=f(Z)$ and for each $e$  define
\begin{eqnarray*}
w_e&=&\mu^{\ast} \big(\{X\ |\ \lim_s \VV(X\restr_s)=e\}\big)\\
w_e[n]&=&\mu^{\ast} \big(\{\sigma\in 2^{n}\ |\ \VV(\sigma)=e\}\big)
\end{eqnarray*}
%
Since $Z\in \DDast$ it follows that $\VV$ learns $\mu^{\ast}$.
Hence
the $\mu^{\ast}$-measure of all the reals $X$ such that 
$ \lim_s \VV(X\restr_s)$ exists and equals an index of
a measure with respect to which $X$ is random, is 1.
If we take into account that
$f(\DD)$ is effectively orthogonal, it follows that 
\begin{equation}\label{gjDmWibB}
\parbox{13cm}{the
$\mu^{\ast}$-measure of all the reals $X$ such that 
$ \lim_s \VV(X\restr_s)$ exists and equals an index of $\mu^{\ast}$ is 1.
Hence there exists an index $t$ of $\mu^{\ast}=f(Z)$
such that $w_t>0$, and moreover,  each $e$ with
$w_e>0$ is an index of $\mu^{\ast}$.}
\end{equation}

\begin{lem}\label{mfqpAdyGd3}
For each $e$, $\lim_n w_e[n]=w_e$.
\end{lem}
\begin{proof}
Since $\VV$ learns $\mu^{\ast}$,
the $\mu^{\ast}$-measure of the reals on which $\VV$ reaches a limit is 1.
For each $n$ let $Q_n$ be the open set of reals on which $\VV$ changes value
after $n$ bits. Then $Q_{n+1}\subseteq Q_n$ and $\lim_n \mu^{\ast} (Q_n)=
\mu^{\ast} (\cap_n Q_n)=0$.
Let $P_e[n]$ be the closed set for reals $X$ such that $\VV(X\restr_i)=e$
for all $i\geq n$. Then $P_e[n]\subseteq P_e[n+1]$ for all $n$ and
$w_e$ is the $\mu^{\ast}$-measure of $\cup_n P_e[n]$.
Hence $w_e=\lim_n \mu^{\ast}(P_e[n])$.

Given $n_0$, for each $n\geq n_0$ we have
$w_e[n]\leq \mu^{\ast}(P_e[n_0]) + \mu^{\ast}(Q_{n_0})$.
This shows that $\limsup_n w_e[n]\leq \limsup_n \mu^{\ast}(P_e[n])=w_e$.
On the other hand $P_e[n_0]\subseteq\dbra{\{\sigma\in 2^{n}\ |\ \VV(\sigma)=e)\}}$
for all $n\geq n_0$. So $w_e=\lim_n ( \mu^{\ast}(P_e[n])) \leq \liminf_n w_e[n]$.
It follows that $\lim_n w_e[n]=w_e$.
\end{proof}

Now, given $Z$ consider
the sequence of computable measures $\mu_{Z\restr_{h(n)}}\in f^{\ast}(Z\restr_{h(n)})$
that are defined by the function $\sigma\mapsto\mu_{\sigma}$ applied on $Z$,
and let 
\[
w^{\ast}_e[n]=\mu_{Z\restr_{h(n)}} \big(\{\sigma\in 2^{n}\ |\ \VV(\sigma)=e\}\big).
\]
From \eqref{S72AlHj9Wf} we get that for each $n,e$,
\begin{equation}
|w_e[n]-w^{\ast}_e[n]|< 2^{-n}.
\end{equation}
In particular, by Lemma \ref{mfqpAdyGd3},
$w_e=\lim_n w_e[n]=\lim_n w^{\ast}_e[n]$.
Let $m$ be some index such that $w_m = \max_e w_e$.
\begin{lem}\label{MsTBC1NH4a}
There exists $n_0$ such that for all $n\ge n_0$ and all $e$, $|w_e - w^{\ast}_e[n]|< w_m/5$.
\end{lem}

\begin{proof}
By~\eqref{gjDmWibB} we have $\sum_e w_e =1$ and $0 < w_m \le 1$.
Then there exists $e_0$ such that $\sum_{e < e_0} w_e > 1 - w_m/20$.
Since we also have for all $e$, $\lim_n w^{\ast}_e[n] = w_e$, 
there exists $n_0$ such that for all $n\ge n_0$, $\sum_{e < e_0} |w_e - w^{\ast}_e[n]|< w_m/20$.
Then for $e < e_0$, it is clear that for all $n\ge n_0$, $|w_e - w^{\ast}_e[n]|< w_m/5$.
On the other hand, we have $\sum_{e \ge e_0} w_e = 1- \sum_{e < e_0} w_e < w_m/20$.
Moreover, for all $n\ge n_0$, 
\[
\sum_{e \ge e_0} w^{\ast}_e[n] = 1- \sum_{e < e_0} w^{\ast}_e[n] \le 
1 - \Big(\sum_{e < e_0} w_e  - w_m/20\Big) <  w_m/10.
\]
So for all $e \ge e_0$, $0 \le w_e< w_m/20$ and $0 \le w^{\ast}_e[n] < w_m/10$, and thus, $|w_e - w^{\ast}_e[n]|< w_m/5$.
\end{proof}

Let us now fix the constant $n_0$ of Lemma \ref{MsTBC1NH4a}.

\begin{lem}[The limit exists]
The value of $\mathcal{L}^{\ast}(Z\upharpoonright _n)$ will converge to some index $i$ with $w_i>0$.
\end{lem}
\begin{proof}
Let $\mathcal{L}^{\ast}(Z\upharpoonright _{h(n_0)}) = e_0$.
In the case that there is some $n\ge h(n_0)$ with 
$\mathcal{L}^{\ast}(Z\upharpoonright _n) \neq e_0$,
there is some $n_1\ge n_0$ such that 
$\mathcal{L}^{\ast}(Z\upharpoonright _{h(n_1)}) = e_1 \neq e_0$.
It then follows from the construction of $\mathcal{L}^{\ast} $ that $w^{\ast}_{e_1}[n_1] \ge w^{\ast}_m[n_1] > 4w_m/5$.
Then by Lemma~\ref{MsTBC1NH4a} for all $n\ge n_1$, 
\[
w^{\ast}_{e_1}[n] > w_{e_1} - w_m/5 > w^{\ast}_{e_1}[n_1] - 2w_m/5 = 2w_m/5
\] 
and on the other hand for all $e$, $w^{\ast}_e[n] < w_e + w_m/5 \le 6w_m/5 < 3w^{\ast}_{e_1}[n]$.
This means that after step $n_1$ the value of $\mathcal{L}^{\ast}(Z\upharpoonright _n)$ will not change and thus, $\lim _n \mathcal{L}^{\ast}(Z\upharpoonright _n) = e_1$ and $w_{e_1} > 4w_m/5 > 0$.
In the case that for all $n\ge h(n_0)$ we have $\mathcal{L}^{\ast}(Z\upharpoonright _n) = e_0$,  
we only need to show that $w_{e_0} > 0$.
Assuming that $w_{e_0} = 0$, there will be some $n_2 > n_0$ such that for all $n \ge n_2$, $w^{\ast}_{e_0}[n] < w_m/4$. Note that $w^{\ast}_m[n] > 4w_m/5 > 3w^{\ast}_{e_0}[n]$, so by the construction of $\mathcal{L}^{\ast}$ the value of $\mathcal{L}^{\ast}(Z\upharpoonright _{h(n_2)})$ will change. 
This is a contradiction.
\end{proof}
The above lemma together with \eqref{gjDmWibB} concludes the proof of \eqref{oPZCcbP3K} and 
 the `only if' direction of Theorem \ref{d3lcvKUJi}.

\subsection{From learning measures to learning reals: the BC case}\label{EZtHGpioBE}
We show the `if' direction of the BC case of Theorem \ref{d3lcvKUJi}.
So consider 
$f:\twome\to\MM$, 
$\DD,\DDast\subseteq \twome$ as given and assume that
$f(\DDast)$ is a BC-learnable class of computable measures.
This means that there exists a learner $\VV$ such that
for all $\mu\in  f(\DDast)$ and $\mu$-random $X$ 
\begin{equation}\label{fFTj1DnwC}
\parbox{13cm}{there exists some $s_0$
such that for all $s>s_0$ the value of $\VV(X\restr_s)$
is an index of $\mu$.}
\end{equation}
We use the expression
$\lim_n \VV(X\restr_n)\approx \mu$  in order to denote
property \eqref{fFTj1DnwC}. Hence our hypothesis about $\VV$ is 
\begin{equation}
\parbox{13cm}{for all $\mu\in  f(\DDast)$ and $\mu$-random $X$ we have
$\lim_n \VV(X\restr_n)\approx \mu$.}
\end{equation}
{\bf Proof idea.}
We would like to employ some kind of majority argument as we did in
Section \ref{g2uXHFLpP}. The problem is that now, given $Z\in\DDast$,  
there is no way to
assign weight on the various indices suggested by $\VV$, in a way that
this weight can be consistently approximated. The reason for this is that
$\VV$ is only a BC-learner and at each step the index guesses 
along  the random reals with respect to $\mu^{\ast}=f(Z)$
may change.  However there is a convergence
in terms of the actual measures that the various indices represent, so we use a function
that takes any number of indices, and as long as there is a majority with respect to
the measures that these indices describe, it outputs an index of this majority  measure.
With this modification, the rest of the argument follows the structure of Section 
 \ref{g2uXHFLpP}.
 
{\bf The formal argument.}

\begin{defi}[Weighted sets]\label{LgRmSqQZrY}
A weighted set is a finite set $A\subset\Nat$ along with
a computable function $(i,s)\mapsto w_i[s]$ from $A\times\Nat$ to the dyadic rationals 
such that $w_i[s]\leq w_i[s+1]$ and $\sum_{i\in A} w_i[s]\leq 1$ for all $s$. 
Given such a weighted set, the weight
of any subset $B\subseteq A$ is $\sum_{i\in B} w_i$, where $w_i:=\lim_s w_i[s]$.\footnote{It
follows from Definition \ref{LgRmSqQZrY} that there is a uniform enumeration of all
weighted sets as programs, so we may refer to {\em an index of a weighted set}. Just like
in any uniform enumeration of programs, we can fix a numbering such that 
any $e\in\Nat$ may be regarded
as an index of a weighted set.}
\end{defi}
In the following we regard each partial computable measure $\mu_e$ as a 
\ce set of tuples $(\sigma,I)$ such that $I$ is a basic open set of $[0,1]$ 
and $\mu_e(\sigma)\in I$
(see Section \ref{Zs7iCd895A}).
\begin{defi}[Majority measures]\label{iBbTCxQ9r5}
Given a weighted set $A$ and a partial computable measure $\mu$,
if the weight of $A\cap\{e\ |\ \mu_e=\mu\}$ is more than $1/2$ we say that
$\mu$ is the majority partial computable measure of $A$.
\end{defi}
Note that there can be at most one 
majority partial computable measure of a weighted set.
In the case that $\mu$ of Definition \ref{iBbTCxQ9r5}
is total, we call it the majority measure of $A$.
\begin{lem}\label{wplqf31FaX}
There is a computable function that maps any index of a weighted set $A$
to an index of a partial computable measure $\mu$ with the property that if
$A$ has a majority partial computable 
measure $\nu$ then $\mu=\nu$.\footnote{More formally, there exists a computable
function $g:\Nat\to\Nat$ such that for each index $e$ of a weighted set $A$,
$g(e)$ is an index of a partial computable measure, and in the case that $A$
has a majority partial computable measure $\nu$, $\mu_{g(e)}=\nu$.}
\end{lem}
\begin{proof}
Given a weighted set $A$ we effectively define a partial computable
measure $\mu$ and then verify its properties.
We view partial computable measures as \ce sets of tuples
$(\sigma,I)$  where $\sigma\in\twomel$ and $I$ is an open rational interval of $[0,1]$
and $(\sigma,I)\in\mu$ indicates that $\mu(\sigma)\in I$.
Define the weight of tuple  $(\sigma,I)$ to be the weight of
$\{i\in A\colon (\sigma,I)\in\mu_i\}$. Then define $\mu$
as the tuples $(\sigma,I)$ of weight $>1/2$.

It remains to verify that if $A$ has a majority partial computable measure then
$\mu$ is the majority partial computable measure of $A$. If 
$\nu$ is the majority partial computable measure of $A$ it is clear that
for each $(\sigma,I)\in\nu$ we have $(\sigma,I)\in\mu$.
Conversely, if $(\sigma,I)\in\mu$, there would be a subset $B\subseteq A$
of weight $>1/2$ such that $(\sigma,I)\in\mu_i$ for all $i\in B$. Since
$\nu$ is the majority partial computable measure of $A$, it follows that
there is an index of $\nu$ in $B$ (otherwise the weight of $A$ would exceed 1).
Hence $(\sigma,I)\in\nu$, which concludes the proof.
\end{proof}
Recall the function $g$ from \eqref{xD7HzXSEpd}. It suffices to show that
\begin{equation}\label{ewrFCFgd7u}
\parbox{13cm}{there exists a computable function $\LLast:\twomel\to\twomel$
such that 
for each $Z\in \DD^{\ast}$ we have 
$\lim_s \LLast(Z\restr_s)\approx f(Z)$}
\end{equation}
because then the function $\LL(\sigma)=g(\LLast(\sigma))$ will be a computable
BC-learner for $\DDast$.

{\bf Definition of $\LLast$.}
We let $\LLast$ map the empty string to index 0 and for every other string $\sigma$ 
we define $\LLast(\sigma)$ as follows. If $\sigma\not\in \{2^{h(n)} \ |\ n\in\Nat\}$
 then let $\LLast(\sigma)=\LLast(\tau)$
where $\tau$ is the longest prefix of $\sigma$ in $\{2^{h(n)} \ |\ n\in\Nat\}$.
So it remains to define $\LLast$ in steps, where at step $n$
we define $\LLast$ on all string $\sigma\in 2^{h(n)}$.
Since $f^{\ast}(\sigma)$ is basic open interval in $\MM$, we may use
\eqref{LhxifOccwG} in order to get a
computable function $\sigma\to\mu_{\sigma}$
from strings to computable measures, such that
for each $\sigma$ the measure $\mu_{\sigma}$ belongs to
$f^{\ast}(\sigma)$.
Given $n$ and $\sigma\in 2^{h(n)}$, 
for each $e$ define 
\[
\wgt{e}=\mu_{\sigma}(\{\tau\in 2^{n}\ |\ \VV(\tau)=e\}).
\]
Let $A_n$ be the weighted set of all $e$ such that 
$\wgt{e}>0$ (clearly there are at most $2^n$ many such numbers $e$)
where the weight of $e\in A_n$ is $\wgt{e}$.
Then apply the computable function of
Lemma \ref{wplqf31FaX} to $A_n$ and let $\LLast(\sigma)$ be the resulting index.

{\bf Properties of $\LLast$.}
We show that $\LLast$ satisfies \eqref{ewrFCFgd7u}, so let $Z$ be a computable
member of $\DD^{\ast} $.

Let $\mu^{\ast}:=f(Z)$ and  define
\begin{eqnarray*}
w&=&\mu^{\ast} \Big(\big\{X\ |\ \lim_s (\VV(X\restr_s)\approx\mu^{\ast})\big\}\Big)\\
w_n&=&\mu^{\ast} \Big(\big\{\sigma\in 2^{n}\ |\ 
\textrm{$\VV(\sigma)$ is an index of $\mu^{\ast}$})\big\}\Big).
\end{eqnarray*}
\begin{lem}\label{cxkCAxhJH1}
$\lim_n w_n=w=1$.
\end{lem}
\begin{proof}
Since $Z\in \DD^{\ast} $ it follows that $\VV$ learns $\mu^{\ast}$, hence $w=1$.
It remains to show that $\lim_n w_n=w$.
For each $n$ let $Q_n$ be the open set of reals $X$ with
the property that there exists some 
$t>n$ such that $\VV(X\restr_t)$ is not an index of $\mu^{\ast}$. 
Then $Q_{n+1}\subseteq Q_n$ and since $\VV$ learns $\mu^{\ast}$ we have 
$\lim_n \mu^{\ast} (Q_n)=\mu^{\ast} (\cap_n Q_n)=0$.
Let $P_n$ be the closed set for reals $X$ such that for all $i\geq n$
the value of $\VV(X\restr_i)$ is an index of $\mu^{\ast}$.
Then $P_n\subseteq P_{n+1}$ for all $n$ and
$w$ is the $\mu^{\ast}$-measure of $\cup_n P_n$.
Hence $w=\mu^{\ast}(\cup_n P_n)=\lim_n \mu^{\ast}(P_n)$.

Given $n_0$, for each $n\geq n_0$ we have
$w_n\leq \mu^{\ast}(P_{n_0}) + \mu^{\ast}(Q_{n_0})$.
This shows that $\limsup_n w_n\leq \limsup_n \mu^{\ast}(P_n)=w_e$.
On the other hand $P_{n_0}\subseteq\dbra{\{\sigma\in 2^{n}\ |\ \textrm{$\VV(\sigma)$ is an index of $\mu^{\ast}$})\}}$
for all $n\geq n_0$. So $w=\lim_n ( \mu^{\ast}(P_n)) \leq \liminf_n w_n$.
It follows that $\lim_n w_n=w$.
\end{proof}

\begin{lem}\label{HrqOe7iOop}
For each $Z\in\DDast$, there exists some $n_0$ such that for all $n>n_0$ the value of
$\LLast(Z\restr_n)$ is an index of $f(Z)=\mu^{\ast}$.
\end{lem}
\begin{proof}
Given $Z\in\DDast$ consider the definition of $\LLast(Z\restr_{h(n)})$ during the various
stages $n$, and the associated weighted sets $A_n$ .
According to the construction of $\LLast$ and Lemma \ref{wplqf31FaX} it suffices
to show that
\begin{equation}\label{KrjvUy8FXr}
\parbox{13cm}{there exists $n_0$ such that for all $n>n_0$ the weighted set
$A_n$ in the definition of 
$\LLast(Z\restr_{h(n)})$
has a majority measure which equals $\mu^{\ast}$.}
\end{equation}
Consider
the sequence $\mu_{Z\restr_{h(n)}}\in f^{\ast}(Z\restr_{h(n)})$ 
of computable measures 
that are defined by the function $\sigma\mapsto\mu_{\sigma}$ applied on $Z$,
and let 
\[
w^{\ast}_n=\mu_{Z\restr_{h(n)}} \Big(\big\{\sigma\in 2^{n}\ |\ 
\textrm{$\VV(\sigma)$ is an index of $\mu^{\ast}$})\big\}\Big).
\]
From \eqref{S72AlHj9Wf} we get that for each $n$,
$|w_n-w^{\ast}_n|< 2^{-n}$.
In particular, by Lemma \ref{cxkCAxhJH1},
$\lim_n w_n=\lim_n w^{\ast}_n=1$.
For  \eqref{KrjvUy8FXr} it suffices to consider any $n_0$ such that
for all $n>n_0$ we have $w_n^{\ast}>1/2$.
Then by the construction of $\LLast$ at step $n$ and the definition 
of $w_n^{\ast}$ it follows that for each $n>n_0$, the majority measure of the weighted set 
$A_n$ is $\mu^{\ast}$.
\end{proof}

Lemma \ref{HrqOe7iOop} shows that $\LLast$
satisfies \eqref{ewrFCFgd7u}, which concludes the BC case of the proof
of the `if' direction of Theorem \ref{d3lcvKUJi}.

\subsection{From learning measures to learning reals: an extension}\label{F8FbxE437v}
There is a way in which we can relax the hypotheses of the
`if' direction of Theorem \ref{d3lcvKUJi} for EX-learning, which concerns the strength of learning
as well as the orthogonality hypothesis.

\begin{defi}[Partial EX-learnability of classes of computable measures]\label{w2WqzOBoC5}
A class $\CC$ of computable measures is partially EX-learnable if there exists
a computable learner $\VV:\twomel\to\Nat$ such that
\begin{enumerate}[\hspace{0.5cm}(a)]
\item $\CC$ is weakly EX-learnable via $\VV$ (recall Definition \ref{Vr7RSPbOiE});
\item for every $\mu\in\CC$ there exists a $\mu$-random $X$ such that  
$\lim_n \VV(X\restr_n)$ is an index of $\mu$.
\end{enumerate}
\end{defi}
The idea behind this notion is that not only for each $\mu\in \CC$
the learner eventually guesses a correct measure (possibly outside $\CC$) 
along each $\mu$-random real, but in addition every measure $\mu\in\CC$ is
represented as a response of the learner along {\em some} $\mu$-random real.
\begin{thm}[An extension]\label{2dbc2Z3Im}
Suppose that a
computable function $f:\twome\to\MM$ is injective on
an effectively closed set $\DD\subseteq \twome$, and $\DDast\subseteq\DD$ is a set of
computable reals.
If $f(\DDast)$ is a partially EX-learnable class of computable measures then $\DDast$
is an EX-learnable class of computable reals.
\end{thm}

{\bf Proof idea.}
We would like to follow the argument of Section \ref{g2uXHFLpP}, but now we have
a weaker assumption which allows the possibility that given $Z\in\DDast$, $\mu^{\ast}=f(Z)$,
there are indices $e$ with positive weight, which do not describe $\mu^{\ast}$.
In order to eliminate these guesses from the approximation $n\to\LLast(Z\restr_n)$ to 
an index of $f(Z)$, we compare how near the candidate measures are to
our current approximation to $\mu^{\ast}$. Using this approach, combined with the
crucial fact (to be proved) that indices with positive weight correspond to total
measures, allows us to eliminate the incorrect total measures (eventually they will be contained
in basic open sets that are disjoint from the open ball $f(Z\restr_n)$ containing $f(Z)$) 
and correctly approximate an index of $\muast$.

{\bf The formal argument.}
Recall the argument from Section \ref{g2uXHFLpP} and note that
\eqref{xD7HzXSEpd} continues to hold under the hypotheses of Theorem \ref{2dbc2Z3Im}.
Hence it suffices to construct a computable $\LLast: \twomel\to\Nat$
such that \eqref{oPZCcbP3K} holds.
Since $f(\DDast)$ is a partially EX-learnable class of computable measures, there exists
$\VV$ with the properties of Definition \ref{w2WqzOBoC5} with respect to $\CC:=f(\DDast)$.
\begin{lem}\label{AJBtUQUdkf}
Every measure $\mu^{\ast}\in f(\DDast)$ has an index $e$ such that
 $\lim_n \VV(X\restr_n)=e$ for a positive $\mu^{\ast}$-measure of reals $X$.
\end{lem}
\begin{proof}
Let $\mu^{\ast}\in f(\DDast)$ and consider a $\mu^{\ast}$-random $X$ such that
$\lim_n \VV(X\restr_n)$ is an index $e$ of $\mu^{\ast}$. Consider
the $\Sigma^0_2$ class $\FF$ of reals $Z$ with the property that
$\lim_n \VV(Z\restr_n)=e$. It remains to show that $\mu^{\ast}(\FF)>0$.
Since $\FF$ is the union of a sequence of $\Pi^0_1$ classes and $X\in \FF$,
there exists a $\Pi^0_1$ class $P\subseteq \FF$ which contains $X$.
Since $X$ is $\mu^{\ast} $-random, it follows that $\mu^{\ast}(P)>0$, so $\mu^{\ast}(\FF)\geq \mu^{\ast}(P)>0$.
\end{proof}

Given $\mu^{\ast}\in f(\DDast)$  define 
$w_e,w_e[n]$
as we did in Section \ref{g2uXHFLpP}.
Note that Lemma \ref{mfqpAdyGd3} still holds by the same argument, since it 
only uses the hypotheses we presently have about $\DD, f,\VV$.

\begin{lem}\label{FDNNq7nVa5}
For every $\mu^{\ast}\in f(\DDast)$ there exists an index $e$ of $\muast$ such that $w_e>0$.
Conversely, if $w_e>0$ then $e$ is an index of a computable measure $\mu'$.
\end{lem}
\begin{proof}
The first claim is Lemma \ref{AJBtUQUdkf}.
For the second claim, if $w_e>0$ it follows from clause (a) of Definition \ref{w2WqzOBoC5}
applied on $\VV$ that $e$ is an index of a computable measure $\mu'$ such that 
all reals in some set $Q$ with $\muast(Q)=w_e>0$ are $\mu'$-random.
\end{proof}

Let $H$ be a partial computable predicate such that for every basic open set $B$ of $\MM$
and every $e$ such that $\mu_e$ is total, we have $H(B,e)\de$ if and only if
$\mu_e\not\in B$.\footnote{The machine for $H$ starts producing a
sequence of basic open sets $A_s$ converging to $\mu_e$ based on the program $e$,
and stops at the first stage $s$ such that $B\cap A_s=\emptyset$, at which point
it halts.}
Hence
\begin{equation}\label{fOPuEjTPB5}
\parbox{13cm}{if $\mu_e$ is total then,
$\exists n \ H(f^{\ast}(X\restr_n),e)[n]\de\iff\mu_e\neq \lim_n f^{\ast}(X\restr_n)$.}
\end{equation}
where the suffix `[n]' indicates the state of $H$ after $n$ steps of computation.

\textbf{Construction of $\LLast$.}
We let $\LLast$ map the empty string to index 0 and for every other string $\sigma$ 
we define $\LLast(\sigma)$ as follows. If $\sigma\not\in \{2^{h(n)} \ |\ n\in\Nat\}$
 then let $\LLast(\sigma)=\LLast(\tau)$
where $\tau$ is the longest prefix of $\sigma$ in $\{2^{h(n)} \ |\ n\in\Nat\}$.
So it remains to define $\LLast$ in steps, where at step $n$
we define $\LLast$ on all string $\sigma\in 2^{h(n)}$.
Since $f^{\ast}(\sigma)$ is basic open interval in $\MM$ so we may use
\eqref{LhxifOccwG} in order to get a
computable function $\sigma\to\mu_{\sigma}$
from strings to computable measures, such that
for each $\sigma$ the measure $\mu_{\sigma}$ belongs to
$f^{\ast}(\sigma)$.

Given $n$ and $\sigma\in 2^{h(n)}$, 
for each $e$ define 
$\wgt{e}=\mu_{\sigma}(\{\tau\in 2^{n}\ |\ \VV(\tau)=e\})$.
Let $\sigma^-$ denote the first $|\sigma|-1$ bits of  $\sigma$
and define $t=\LLast(\sigma^-)$.
Let $e^{\ast}$ be the least number with the maximum $\wgt{e}$ such that
$H(f^{\ast}(\sigma),\mu_e)[n]\un$; if this does not exist, define
$\LLast(\sigma)=\LLast(\sigma^-)$.
Otherwise, if
one of the following holds
\begin{enumerate}[\hspace{0.5cm}(a)]
\item $\wgt{e^{\ast}}>3\cdot \wgt{t}$  
\item $H(f^{\ast}(\sigma),t)[n]\de$
\end{enumerate}
let $\LLast(\sigma)=e^{\ast}$.  In any other case let $\LLast(\sigma)=\LLast(\sigma^-)$.

{\bf Properties of $\LLast$.}
We show that  \eqref{oPZCcbP3K} holds, \ie that
for each $Z\in \DDast$ the limit 
$\lim_s \LLast(Z\restr_s)$ exists and is an index for $f(Z)$.
Let $Z\in\DDast$, $\muast=f(Z)$ and consider
the sequence of computable measures $\mu_{Z\restr_{h(n)}}\in f^{\ast}(Z\restr_{h(n)})$
that are defined by the function $\sigma\mapsto\mu_{\sigma}$ applied on $Z$, and are
used in the steps $n$ of the definition of $\LLast$ with respect to $Z$.
Let
\[
w^{\ast}_e[n]=\mu_{Z\restr_{h(n)}} \big(\{\sigma\in 2^{n}\ |\ \VV(\sigma)=e)\}\big).
\]
and note that these are the weights used in the definition of $\LLast$ at step $n$
with respect to $Z\restr_{h(n)}$.

\begin{lem}\label{9yjxmo7eek}
For each $e$, 
$w_e=\lim_n w_e[n]=\lim_n w^{\ast}_e[n]$.
\end{lem}
\begin{proof}
The first equality is Lemma \ref{mfqpAdyGd3}.
From \eqref{S72AlHj9Wf} we get that for each $n,e$,
$|w_e[n]-w^{\ast}_e[n]|< 2^{-n}$, which establishes the second equality.
\end{proof}

Next, we show that  $\lim_s \LLast(Z\restr_s)$ exists.
Let $H_e[n]$ denote $H(f^{\ast}(Z\upharpoonright_{h(n)}), e)[n]$. 
Let $T = \{e\colon e \text{ is an index of } \mu^{\ast} \} $
 and $m$ be some index such that $w_m = \max\{w_e\colon e \in T \}$.
By Lemma 3.11 $w_m > 0$.
By \eqref{fOPuEjTPB5}, $e \in T$ if and only if for all $n$, $H_e[n] \uparrow$.
\begin{lem}\label{Z4z2obsXy}
There exists $n_0$ such that for all $n\ge n_0$ and all $e$, 
\begin{enumerate}[\hspace{0.5cm}(i)]
\item $|w_e - w^{\ast}_e[n]|< w_m/5$. 
\item If $w_e > 4 w_m/5$ and $e \notin T$ then $H_e [n] \downarrow$.
\end{enumerate}
\end{lem}
\begin{proof}
Since $\sum_e w_e =1$ and $0 < w_m \le 1$, then there exists $e_0$ 
such that $\sum_{e < e_0} w_e > 1 - w_m/20$.
We also have that for all $e$, $\lim_n w^{\ast}_e[n] = w_e$, 
so there exists 
$n^{\ast}$ such that for all $n\ge n^{\ast}$, $\sum_{e < e_0} |w_e - w^{\ast}_e[n]|< w_m/20$.
Then for $e < e_0$, it is clear that for all $n\ge n^{\ast}$, $|w_e - w^{\ast}_e[n]|< w_m/5$.
On the other hand, we have $\sum_{e \ge e_0} w_e = 1- \sum_{e < e_0} w_e < w_m/20$.
Moreover, for all $n\ge n^{\ast}$, 
\[
\sum_{e \ge e_0} w^{\ast}_e[n] = 1- \sum_{e < e_0} w^{\ast}_e[n] \le 1 - 
\Big(\sum_{e < e_0} w_e  - w_m/20\Big) <  w_m/10.
\]
So for all $e \ge e_0$, $0 \le w_e< w_m/20$ and $0 \le w^{\ast}_e[n] < w_m/10$, and thus, $|w_e - w^{\ast}_e[n]|< w_m/5$.
If $w_e > 4 w_m/5$, 
it must be case that $e < e_0$, and thus, there are only finitely many such indices $e$.
For every such index $e$, if $e\notin T $, there will be some 
$k_e$ such that for all $n \ge k_e$, $H_e [n] \downarrow$.
If $n_0$ is the largest number amongst $k_e$ and $n^{\ast}$, clauses (i) and (ii) in the
statement of the lemma hold. 
\end{proof}

Let us now fix the constant $n_0$ of Lemma \ref{Z4z2obsXy}.

\begin{lem}[The limit exists]\label{onLX1pUewW}
The value of $\mathcal{L}^{\ast}(Z\upharpoonright _n)$ will converge to some index $i\in T$. 
\end{lem}
\begin{proof}
Let $\mathcal{L}^{\ast}(Z\upharpoonright _{h(n_0)}) = e_0 $.
In the case that there is some $n\ge h(n_0)$ with $\mathcal{L}^{\ast}(Z\upharpoonright _{n}) \neq e_0$, there should be some $n_1 > n_0$ 
such that $\mathcal{L}^{\ast}(Z\upharpoonright _{h(n_1)}) = e_1 \neq e_0$.
It then follows from the construction of $\mathcal{L}^{\ast} $ that $w^{\ast}_{e_1}[n_1] \ge w^{\ast}_m[n_1] > 4w_m/5$ and $H_{e_1}[n_1] \uparrow $.
Then by Lemma~\ref{Z4z2obsXy} for all $n\ge n_1$, 
\[
w^{\ast}_{e_1}[n] > w_{e_1} - w_m/5 > w^{\ast}_{e_1}[n_1] - 2w_m/5 = 2w_m/5
\] 
and $e_1\in T$.
On the other hand if $e \in T$, 
for all $n\ge n_1$ we have $w^{\ast}_e[n] < w_e + w_m/5 \le 6w_m/5 < 3w^{\ast}_{e_1}[n]$.
If $e \notin T$ but $w_e[n] > 6 w_m/5$, then $w_e > w_m$, 
so for all $n\ge n_1$ we have $H_e [n] \downarrow$.
This means that after step $n_1$ the value of $\mathcal{L}^{\ast}(Z\upharpoonright _n)$ will not change and thus, $\lim _n \mathcal{L}^{\ast}(Z\upharpoonright _n) = e_1 \in T$.
In the case that for all $n\ge h(n_0)$ we have $\mathcal{L}^{\ast}(Z\upharpoonright _n) = e_0$,  
we only need to show that $e_0 \in T$.
Assuming that $e_0 \notin T$, there exists some step $n_2\ge n_0$ such that 
$H_{e_0} [n_2]\downarrow$.
Since $m\in T$, for all  $n\ge n_0$ we have $H_m [n]\uparrow$.
By the construction of $\mathcal{L}^{\ast}$ the value of 
$\mathcal{L}^{\ast}(Z\upharpoonright _{h(n_2)})$ will change. This is a contradiction.
\end{proof}

The above lemma concludes the proof of Theorem \ref{2dbc2Z3Im}.

\subsection{Proof of Theorem \ref{WVmrKGDjbs}}\label{In6hhYHW}
It is well known that if $Z$ is computable and $\mu$-random 
for some computable measure $\mu$,
then $Z$ is an atom of $\mu$ and $\mu(Z\restr_n\ast Z(n))/\mu(Z\restr_n)$ tends to 1. Here is a generalization.

\begin{lem}\label{UfrLEbKHaN}
If $Z$ is computable and
$Z\oplus Y$ is $\mu$-random for some computable measure $\mu$,
then 
$\mu(Z\restr_n\oplus Y\restr_n \ast Z(n))/\mu(Z\restr_n\oplus Y\restr_n)\to 1$
as $n\to\infty$.
\end{lem}
\begin{proof}
We prove the contrapositive: fix computable $\mu,Z$, and suppose that for some $Y$
there exists a rational $q\in (0,1)$ such that
\begin{equation}\label{2K8zd6Q3sC}
\mu((Z\restr_n\oplus Y\restr_n) \ast Z(n))<q\cdot \mu(Z\restr_n\oplus Y\restr_n)
\end{equation}
for infinitely many $n$.
For each $t$ consider the set $V_t$ of the strings
of the form $(Z\restr_j\oplus X\restr_j)\ast Z(j)$ for some $j,X$, such that $j$ is minimal
with the property that there exist at least $t$ many $n\leq j$ with \eqref{2K8zd6Q3sC} by replacing $Y$ with $X$.
For each 
nonempty string
$\sigma$, let $\sigma^{-}$ denote the largest proper prefix of $\sigma$.
By the minimality of the choice of $n$ above, we have that  (a) $V_t$ is prefix-free;
(b) each
string $\tau\in V_{t+1}$ extends a string $\sigma\in V_t$;
(c) if $\sigma\in V_t$ then $\mu(\sigma)<q\cdot \mu(\sigma^{-})$;
(d) if $V_{t+1}(\sigma)$ is the set of all the strings in $V_{t+1}$ extending $\sigma\in V_t$ then $ \mu(V_{t+1}(\sigma)) < q\cdot \mu(\sigma)$.
%
It follows that $\mu(V_{t+1})<q\cdot \mu(V_{t})$ so there exists
a computable sequence $(m_j)$ such that $\mu(V_{m_j})<2^{-j}$ for each $j$.
So $(V_{m_j})$ is a $\mu$-test and by its definition, if $Y$ satisfies \eqref{2K8zd6Q3sC}
for infinitely many $n$, then $Z \oplus Y$ has a prefix in $V_t$ for each $t$, and so in 
$V_{m_j}$ for each $j$. Hence in this case $Z\oplus Y$ is not $\mu$-random. 
\end{proof}

For each $Z$ define $\mu_Z$ as follows: for each $\sigma$
of odd length
let $\mu_Z(\sigma\ast i)=\mu_Z(\sigma)/2$ for $i=0,1$; for each
$\sigma$ of even length
let $j_{\sigma}=Z(|\sigma|/2)$ and define 
$\mu_Z(\sigma\ast j_{\sigma})=\mu_Z(\sigma)$,
$\mu_Z(\sigma\ast (1-j_{\sigma}))=0$. Hence for each
real $X$ and each $n$, all  $\mu_Z(X\restr_{2n})$ goes to $X\restr_{2n}\ast Z(n)$
while $\mu_Z(X\restr_{2n+1})$ is split equally to $X\restr_{2n+1}\ast 0$ and $X\restr_{2n+1}\ast 1$.
Note that for each $Z$ the measure $\mu_Z$ is continuous. Also, the map
$Z\mapsto\mu_Z$ from $\twome$ to $\MM$ is continuous.
 
 \begin{lem}\label{CG1aUvOCf}
Given any computable $Z$, a real $X$ is $\mu_Z$-random if and only if
it is of the form $Z\oplus Y$ for some random $Y$ with respect to the uniform measure.
\end{lem}
\begin{proof}
``$\Rightarrow$: ''
If  $X$ is of the form $W\oplus Y$ for some $W\neq Z$ then by the definition of $\mu_Z$
we have
$\mu_Z((W\oplus Y)\restr_n)=0$ for sufficiently large $n$, so
$W\oplus Y$ is not $\mu_Z$-random.
If $X$ is of the form $Z\oplus Y$ and $Y$ is not random with respect to the 
uniform measure $\lambda$,
let $(V_i)$ be a $\lambda$-test such that $Y\in \cap_i\dbra{V_i}$.
For each $i$ let
$U_i=\{Z\restr_{|\sigma|}\oplus\sigma \ |\ \sigma\in V_i\}$. By the definition of $\mu_Z$ we
have $\mu_Z(U_i)=\lambda(V_i)\leq 2^{-i}$ so $(U_i)$ is a $\mu_Z$-test.
Since $Y\in \cap_i\dbra{V_i}$ we have $Z\oplus Y\in  \cap_i\dbra{U_i}$ hence
$Z\oplus Y$ is not $\mu_Z$-random.\\
``$\Leftarrow$: ''
If $Z\oplus Y$ is not $\mu_Z$-random, then there is a $\mu_Z$-test $(U_i)$ such that $Z\oplus Y\in  \cap_i\dbra{U_i}$.
For each $i$ let $V_i=\{\sigma(1)\sigma(3)\cdots\sigma(2n-1)\ |\ \sigma\in U_i \text{ and } n=\lfloor |\sigma|/2\rfloor  \}$.
By the definition of $\mu_Z$ we have $\lambda(V_i) = \mu_Z(U_i) \leq 2^{-i}$ and $Y\in \cap_i\dbra{V_i}$.
So $Y $ is not random with respect to the uniform measure.
\end{proof}
Hence if $Z\neq X$ are computable,  
the measures $\mu_Z, \mu_X$ are effectively orthogonal.
Then the `only if' direction of Theorem \ref{WVmrKGDjbs}
follows from the `only if' direction of  Theorem \ref{d3lcvKUJi}
(with $\DD:=\twome$ and $\DDast:=\CC$). 
The following concludes the proof of Theorem \ref{WVmrKGDjbs}. 
\begin{lem}
For each class $\CC$ of computable reals, if
$\{\mu_Z\ |\ Z\in\CC\}$ is a weakly EX/BC learnable class of measures then 
$\CC$ is EX/BC learnable.
\end{lem}
\begin{proof}
We first show the EX case.
Fix $\CC$ and let $\VV$ be a learner which EX-succeeds on all
measures in $\{\mu_Z\ |\ Z\in\CC\}$. 
It remains to construct an 
EX-learner $\LL$ for $\CC$.

{\bf Proof idea.} 
Given a computable $Z$, in order to define $\LL(Z\restr_n)$ we 
use $\VV$ on the strings $Z\restr_n\oplus \sigma$, $\sigma\in 2^n$ and take
a majority vote in order to determine $Z(n)$. According to Lemmas \ref{UfrLEbKHaN}
and \ref{CG1aUvOCf}, eventually the correct value of $Z(n)$ will be the $j$
such that $(Z\restr_n\oplus \sigma)\ast j$ gets most of the measure on 
$(Z\restr_n\oplus \sigma)$, with respect to any measure correctly guessed by 
$\VV(Z\restr_n\oplus \sigma)$, for the majority of $\sigma\in 2^n$.

{\bf Construction of $\LL$.} 
First, define a computable $g_0:\twomel\to\Nat$ as follows, taking
a majority vote via $\VV$. For each $Z$, $n$ we define $g_0(Z\restr_n)$
to be an index of the following partial computable real $X$.
For each $m< n$ we let $X(m)=Z(m)$. 
If $m\geq n$, suppose inductively that it has already defined $X\restr_m$.
In order to define $X(m)$, it calculates the measure-indices
$\VV(X\restr_m\oplus\sigma)=e$ for all $\sigma\in 2^m$ and waits until, for some $j\in\{0,1\}$, 
at least $2/3$ these partial computable measures $\mu_e$ have the property
$\mu_e((X\restr_n\oplus\sigma)\ast j)\de>\mu_e(X\restr_n\oplus\sigma)/2$.
If and when this happens it defines  $X(m)=j$.

Fix $Z\in\CC$.
By Lemma \ref{UfrLEbKHaN},  if $\VV$ weakly EX-learns $\mu_Z$,
for all  sufficiently large $n$ the value of $g_0(Z\restr_n)$
will be an index of $Z$ (possibly different for each $n$).
In order to produce a stable guess, define the function 
$\LL:\twomel\to\Nat$ as follows. In order to define
$\LL(Z\restr_n)$, consider the least $n_0\leq n$ such that 
\begin{enumerate}[(i)]
\item at least proportion $2/3$
of the strings $\sigma\in 2^{n}$ have not changed their $\VV$-guess since  $n_0$,
\ie $\VV(Z\restr_i\oplus\sigma\restr_i)=\VV(Z\restr_{n_0}\oplus\sigma\restr_{n_0})$
for all integers $i\in (n_0, n]$;
\item no disagreement between $Z\restr_n$ and the reals defined by the indices 
$\LL(Z\restr_i)$, $i\in (n_0,n)$ has appeared up to stage $n$; formally,
if $(\varphi_e[n])$ is the fixed effective list of all partial computable reals at stage $n$ of the universal computation,
there exists no $j<n$ and $i\in (n_0,n)$ such that $Z(j)\neq \varphi_i(j)[n]$.
\end{enumerate}
Then let  $\LL(Z\restr_n)$ be $g_0(Z\restr_{n_0})$.
Given $Z\in\CC$ we have that $\VV$ weakly learns $\mu_Z$, so
$\VV(Z\restr_n\oplus Y\restr_n)$ converges for almost all $Y$
(with respect to the uniform measure). Hence in this case (i)
will cease to apply for large enough $n$.
Moreover, by the properties of $g_0$, clause (ii) will also cease to
apply for sufficiently large $n$. Hence the $n_0$ in the definitions of
$\LL(Z\restr_n)$ will stabilize for large enough $n$, and
$\LL(Z\restr_n)$ will reach the limit $g_0(Z\restr_{n_0})$ which
is an index for $Z$.

For the BC case, assume instead that $\VV$ 
BC-succeeds on all
measures in $\{\mu_Z\ |\ Z\in\CC\}$. 
We define $g_0$ exactly as above, and the BC-learner $\LL$ 
by $\LL(Z\restr_n)=g_0(Z\restr_n)$.
Given $Z\in\CC$ we have that $\VV$ weakly BC-learns $\mu_Z$, so
for almost all $Y$ (with respect to the uniform measure), 
$\VV(Z\restr_n\oplus Y\restr_n)$ eventually outputs indices
of a computable measure $\mu$ (dependent on $Y\restr_n$)
with the property that $\mu((Z\restr_n\oplus Y\restr_n)\ast Z(n))>2/3\cdot 
\mu(Z\restr_n\oplus Y\restr_n)$.
By the definition of $g_0$, 
this means that for sufficiently large $n$, the value of $\LL(Z\restr_n)$
is an index of $Z$. Hence $\LL$ is a BC-learner for $\CC$.
\end{proof}

\subsection{Proof of Theorem \ref{EHY3duonav}}
 A stage $s$ is called $i$-expansionary if 
$\ell_i[t]<\ell_i[s]$ for all  $i$-expansionary stages $t<s$.
By the padding lemma let $p$ be a computable function such that for each $i,j$ 
we have $\mu_{p(i,j)}\simeq \mu_{i}$ and $p(i,j)<p(i,j+1)$.

Define the $e$th 
randomness deficiency function by setting $d_e(\sigma)$  to be 
$\ceil{-\log \mu_e(\sigma)}-K(\sigma)$ for each string $\sigma$, 
where $K$ is the \pf complexity of $\sigma$.
Define the $e$th
randomness deficiency on a real $X$ as:
$\mathbf{d}_e(X)=\sup_n d_e(X\restr_n)$
where the supremum is taken over the $n$ such that 
$d_e(X\restr_n)\de$. By \citep{leviniandc/Levin84}, if $\mu_e$ is total then 
$X$ is $\mu_e$-random if and only if $\mathbf{d}_e(X)<\infty$.

At stage $s$, we define $\LL(\sigma)$ for each $\sigma$ of length $s$ as follows.
For the definition of  $\LL(\sigma)$ find the least $i$ such that $s$ is $i$-expansionary and
$d_i(\sigma)[s]\leq i$. Then let $j$ be the least
such that $p(i,j)$ is larger than any $k$-expansionary stage $t<|\sigma|$ for any $k<i$
such that $d_k(\sigma\restr_k)[t]\leq k$,
and define $\LL(\sigma_t)=p(i,j)$.

Let $X$ be a real. Note that $\LL(X\restr_n)=x$ for infinitely many $n$, then
$x=p(i,j)$ for some $i,j$, which means that $\mu_i=\mu_x$ is total
and there are infinitely many $x$-expansionary stages as well as infinitely many
$i$-expansionary stages.
This implies that there are at most $x$ many 
$y$-expansionary stages $t$ for any
$y<x$ with $d_y(\sigma\restr_y)[t]\leq y$. Moreover, for each $z>x$
there are at most finitely may $n$ such that $\LL(X\restr_n)=z$. Indeed, for each $z$
if $n_0$ is an $i$-expansionary stage then $\LL(X\restr_n)\neq z$ for all $n>n_0$.
Moreover, if  $\LL(X\restr_n)=x$ for infinitely many $n$, then 
$d_x(X)=d_i(X)\leq i$ and $\mu_i$ is total, so $X$ is $\mu_i$-random.
We have shown that for each $X$ there exists at most one $x$ such that
$\LL(X\restr_n)=x$ for infinitely many $n$, and in this case $\mu_x$ is total
and $X$ is $\mu_x$-random.

It remains to show that if $X$ is $\mu$-random for some computable $\mu$, then
there exists some $x$ such that $\LL(X\restr_n)=x$ for infinitely many $n$.
If $X$ is $\mu_i$-random for some $i$ such that $\mu_i$ is total,
let $i$ be the least such number 
with the additional property that $\mathbf{d}_i(X)\leq i$ (which exists by the padding lemma).
Also let $j$ be the least number such that 
$p(i,j)$ is larger than any stage $t$ which is 
$k$-expansionary for any $k<i$
with $d_k(\sigma\restr_k)[t]\leq k$. Then the construction will define
$\LL(X\restr_n)=p(i,j)$ for each $i$-expansionary stage $n$ after
the last $k$-expansionary stage $t$ for any $k<i$
with $d_k(\sigma\restr_k)[t]\leq k$.
We have shown that $\LL$ partially succeeds on every $\mu$-random $X$ for any
computable measure $\mu$.

\section{Applications}\label{P6hQ2fd9BA}
For the `if' direction of Corollaries \ref{gyGJY8fGL} 
and \ref{P2b9kifIJR} we need the following lemma.
\begin{lem}
If $A$ is high then the class of  
all computable measures and the class of
all computable Bernoulli measures are both EX$[A]$-learnable.
\end{lem}
\begin{proof}
We first show the part for the computable Bernoulli measures.
The function which maps a real $X\in\twome$ to the
measure representation $\mu:\twome\to\mathbb{R}$ of the 
Bernoulli measure with success probability the real in $\mathbb{R}$
with binary expansion $X$ is computable.
Hence, given an effective list $(\mu_e)$ of all partial computable measures in $\MM$
and an effective list $(\varphi_e)$ of all partial computable reals in $\twome$,
there exists a computable function $g:\Nat\to\Nat$ such that for each $e$
such that $\varphi_e$ is total, $\mu_{g(e)}$ is total and is the measure representation of
the Bernoulli measure with success probability  the real with binary expansion $\varphi_e$.
Let $\mu_e[s]$ represent the state of approximation to $\mu_e$ at stage $s$ of the universal
approximation, so $\mu_e[s]$ is a basic open set of $\MM$.  Then the function
$\sigma\to\sup\{\mu(\sigma)\ |\ \mu\in \mu_e[s]\}$ is computable and the function
\[
d(e,\sigma)[s]=\ceil{-\log\big(\sup\{\mu(\sigma)\ |\ \mu\in \mu_e[s]\}\big)} - K(\sigma)[s]
\]
is a computable approximation to the $\mu_e$-deficiency of $\sigma$, in the case that
$\mu_e$ is total.
Since $A$ is high there exists a function 
$h:\Nat\times\Nat\to\{0,1\}$, $h\leq_T A$  such that
for each $e$, $\varphi_e$ is total if and only if $\lim_s h(e,s)=1$.
Define $cost(e,\sigma)=e+d(g(e),\sigma)[|\sigma|]$.
We define an $A$-computable learner $\VV$ as follows:
for each $\sigma$ let $\VV(\sigma)$ be $g(e)$ for the least index $e\leq |\sigma|$
which minimizes $cost(e,\sigma)$ subject to the condition $h(e,|\sigma|)=1$.

It remains to show that for each $X\in\twome$ which is random with respect to
a computable Bernoulli measure $\mu$, $\lim_n \VV(X\restr_n)$ exists and equals
an index of $\mu$. According to our working assumption about $X$,
there exist numbers $e$ such that $\varphi_e$ is total and
$\sup_n cost(e,X\restr_n)<\infty$. These numbers $e$ are the indices of reals in $\twome$
which are the binary representations of the success probability of the Bernoulli measure
with respect to which $X$ is random. Now consider the least $e$ with this property,
and which  minimizes $\sup_n cost(e,X\restr_n)$. 
Note that, by the definition of  $cost(e,\sigma)$, for each $k,\sigma$
there are only finitely many $e$ such that $cost(e,\sigma)<k$.
It follows by the construction of
$\VV$ that $\lim_n \VV(X\restr_n)=e$.

The proof for the class of all computable measures is the same as above,
except that we take $g$ to be the identity function.
\end{proof}

For the other direction of Corollary \ref{gyGJY8fGL},
let $\CC$ be the class of all computable reals, and 
assume that the computable measures are
weakly EX$[A]$-learnable. Then $\{\mu_Z\ |\ Z\in\CC\}$ is also
weakly EX$[A]$-learnable, and by Theorem 
\ref{WVmrKGDjbs} we get that $\CC$ is  EX$[A]$-learnable.
Then by \citep{AdlemanB91} it follows that $A$ is high.

\subsection{Applying Theorem \ref{d3lcvKUJi} to classes of Bernoulli measures}
Perhaps the most natural parametrization of measures on $\twome$ by reals
is the following.
\begin{defi}
Consider the function
$f_b:\twome\to\MM$ mapping
each $X\in\twome$ to the
Bernoulli measure with success probability the real whose binary expansion is $X$.
\end{defi}
Clearly $f_b$ is computable, but it is not injective since
dyadic reals have two different binary expansions.
In order to mitigate this inconvenience, we consider the following transformation.
\begin{defi}
Given any $\sigma\in\twomel$ or $X\in\twome$, let $\hat{\sigma},\hat{X}$
be the string or real respectively obtained from $\sigma,X$ by 
the digit replacement $0\to 01$, $1\to 10$.
Fore each class $\CC\subseteq\twome$ let $\hat{\CC}=\{\hat{X}\ |\ X\in\CC\}$.
\end{defi}
Since no real in $\CCh$ is dyadic, $f_b$ is injective on $\CCh$.
Moreover, $\CCh$ has the same effectivity properties as $\CC$;
for example it is effectively closed if and only if $\CC$ is.
Hence
the hypotheses of Theorem \ref{d3lcvKUJi} are satisfied
for $f:=f_b$ and $\DD:=\CCh$ for any effectively closed $\CC\subseteq\twome$.


\begin{lem}[Invariance under computable translation]\label{C9sahblZM3}
A class $\CC\subseteq\twome$ of computable reals is EX-learnable if and only if
the class $\CCh$ is EX-learnable.
The same is true of BC-learnability.
\end{lem}
\begin{proof}
Clearly 
$\CC$, $\CCh$ are computably isomorphic. 
Suppose that $\CC$ is EX or BC learnable by $\LL$.
Let $g$ be a computable function that maps each index $e$ of computable real $X$
to an index $g(e)$ of the computable real $\hat{X}$.
For each $\sigma\in \twome$
define $\LLast(\hat{\sigma})=g(L(\sigma))$. Moreover, for each $\tau$ 
which is a prefix of a real in $\widehat{\twomel}$ but whose length
is not a multiple of 2, define $\LLast(\tau)=\LLast(\rho)$ where $\rho$ is the largest initial 
segment of $\tau$ which is a multiple of 2. If $\tau$ is not a prefix of a real in $\widehat{\twomel}$
then let $\LLast(\tau)=0$.
In the case of EX learning, 
since for each real $X\in\CC$
the values $\LL(X\restr_n)$ converge to an index $e$ of $X$, it follows that
the values $\LLast(\hat{X}\restr_n)$ converge to the index $g(e)$ of $\hat{X}$, so
$\LLast$ is an EX-learner for $\CCh$.
The case for BC learning as well as the converse are entirely similar.
\end{proof}

\begin{lem}\label{lBLzXkW7lq}
A class of computable reals $\CC\subseteq \twome$ is EX-learnable if and only if
the class $f_b(\CCh)$ of Bernoulli measures is EX-learnable.
The same is true for BC learnability.
\end{lem}
\begin{proof}
By Lemma \ref{C9sahblZM3}, $\CC$ is EX-learnable if and only if 
$\CCh$ is. If we consider $\CCh$ as a subset of the effectively closed 
set $\DD=\widehat{\twome}$ and apply  Theorem \ref{d3lcvKUJi} for  $f_b$
we get that $\CCh$ is EX-learnable if and only if  $f_b(\CCh)$ is.
\end{proof}

\subsection{Proofs of the corollaries of Section \ref{Qx1nj3CucT}} \label{qLIwtAZfTH}
We conclude the proof
of Corollary \ref{P2b9kifIJR} by showing that
if an oracle can EX-learn all computable Bernoulli measures then it is high.
Note that 
learnability of an effectively orthogonal class of measures
is closed under subsets. 
Hence it suffices to show that if an oracle $A$
can EX-learn all computable Bernoulli measures with success probabilities
that have a binary expansion in $\widehat{\twome}$,
then it is high.
By a direct relativization of
Theorem  \ref{d3lcvKUJi} and Lemma \ref{lBLzXkW7lq}, 
the above working assumption on $A$ implies that
the class of computable reals is EX-learnable with oracle $A$.
Then by \citep{AdlemanB91}
it follows that $A$ is high. 

Next, we prove Corollary \ref{A3ce2GoRBo}, which says that
there exist two EX-learnable classes of computable (Bernoulli) measures such that
their union is not EX-learnable.
Blum and Blum \citep{BLUM1975125}
defined two classes  $S, T$ of computable functions which are 
EX-learnable but their union is not. 
Consider the classes
$\hat{S}, \hat{T}$, $\hat{S}\cup \hat{T}=\widehat{(S\cup T)}$.
By Corollary \ref{lBLzXkW7lq} the classes $f_b(\hat{S}), f_b(\hat{T})$
are learnable but the class $f_b(\hat{S}\cup \hat{T})$ is not.
The result follows by noticing  that $f_b(\hat{S})\cup f_b(\hat{T})=f_b(\hat{S}\cup \hat{T})$.

Next, we show \eqref{maMViSH2Qa} which says that oracles that are not $\Delta^0_2$ or
are not 1-generic, are not low for EX-learning for measures.
If $A\not\leq_T\emptyset'$
then by  \citep{FORTNOW1994231}
there exists a class $\CC$ of computable reals which is EX$[A]$-learnable but not
EX-learnable.
If $A\leq_T\emptyset'$ and $A$ is not 1-generic,
then by \citep{KUMMER1996214} 
there exists a class  $\CC$ of computable reals which is EX$[A]$-learnable but not
EX-learnable. Then \eqref{maMViSH2Qa} follows by these results, combined with
Corollary \ref{lBLzXkW7lq}.

Finally we prove Corollary
\ref{UT2HmCWUPe}, which says that
a learner can EX-learn all computable measures with finitely many queries
on an oracle $A$ if and only if $\emptyset''\leq_T A\oplus \emptyset'$.
We need the following lemma.
\begin{lem}\label{pvABJZlMm2}
If $A\leq_T B'$ then every class of computable measures which is EX-learnable by $A$ with
finitely many queries, is also EX-learnable by $B$.
\end{lem}
\begin{proof}
This is entirely similar to the analogous result for EX-learning of classes computable reals
from \citep{FORTNOW1994231}.
By $A\leq_T B'$ one can obtain a $B$-computable function that approximates $A$.
Given an $A$-computable learner and replacing the oracle with the approximation given via oracle $B$,
the resulting learner will converge along every real on which the original learner converges
and uses finitely many queries on $A$. Moreover, in this case, the limit will agree with the limit
with respect to the original $A$-computable learner. This shows that
any class that is EX-learnable via the $A$-computable learner will also be EX-learned 
by the new $B$-computable learner.
\end{proof}

Now given an oracle $A$, by the
jump-inversion theorem, since $\emptyset'\leq_T A\oplus \emptyset'$,
there exists some $B$ such that $B'\equiv_T A\oplus \emptyset'$. So $A\leq_T B'$.
By Lemma \ref{pvABJZlMm2}, if the computable measures are EX-learnable 
with oracle $A$ and finitely many queries, then they will also be EX-learnable by $B$.
Then by Corollary \ref{gyGJY8fGL} it follows that 
$B$ is high, so $ B'\geq_T\emptyset''$ and $\emptyset''\leq_T A\oplus \emptyset'$
as required. 

Conversely, assume that $\emptyset''\leq_T A\oplus \emptyset'$. Let
$(\mu_e)$ be a universal enumeration of all partial computable 
measure representations with dyadic values and note that by the discussion of
Section \ref{Zs7iCd895A} it is sufficient to restrict our attention to these measures,
which may not include some measures with non-dyadic values.
By Jocksuch \citep{Jockusch:72*1} there exists a function $h\leq_T A$ such that
$(\mu_{h(e)})$ is a universal enumeration of all {\em total} computable measure
representations with dyadic values. The fact that uniformly computable families
of measures are EX-learnable (originally from \citep{VITANYI201713})
relativizes to any oracle. Since 
$(\mu_{h(e)})$ contains all computable measure representations with dyadic values,
it follows that the class of all computable measures is EX-learnable with oracle $A$.

\section{Conclusion and open questions}
We have presented  tools which allow to transfer many of the results of the theory
of learning of integer functions or reals based on
\citep{GOLD1967447}, to the theory of learning of probability distributions 
which was recently introduced in  \citep{VITANYI201713} and studied in 
\citep{Bienvenu2014,Bienvealgindpro}.
We demonstrated the usefulness of this result with numerous corollaries
that provide parallels between the two learning theories. We also
identified some differences; we found that although
in the special case of effectively orthogonal classes, the notions of 
Definitions \ref{nAemubaidf} and \ref{1PuUYa3Jz4} are closed under the subset relation,
in general they are not so. Intuitively, if we wish to learn a subclass
of a given class of computable measures, the task (compared to learning the original class)
becomes easier 
in one way and harder in another way: it is easier because we only need to consider success
of the learner on $\mu$-random reals for a smaller collection of measures $\mu$; it is
harder because the learner has fewer choices of indices that are correct answers 
along each real, since the class of measures at hand is smaller. 
%
%


We showed that the oracles needed for the EX-learning
of the computable measures are exactly  the oracles needed for
the EX-learning of the computable reals, which are the high oracles.
In the classic theory there exists no succinct characterization of the oracles
that BC-learn the computable functions. On the other hand,
Theorem \ref{WVmrKGDjbs} 
shows that if an oracle can BC-learn the class of computable (continuous)
measures, then it can also BC-learn the class of computable functions.

{\bf Open problem.}  
If an oracle can BC-learn the class of computable functions, is it necessarily
the case that it can learn the class of computable (continuous)
measures?

Another issue discussed is the low for EX-learning oracles for learning of measures.
We showed that every such oracle is also low for EX-learning in the classical learning theory
of reals. We do not know if the converse holds.

\bibliographystyle{abbrvnat}
\bibliography{exmelow}

\end{document}